\documentclass[11pt]{amsart}

\newtheorem{thm}{Theorem}[section]
\newtheorem{prop}[thm]{Proposition}
\newtheorem{conj}[thm]{Conjecture}
\newtheorem{cor}[thm]{Corollary}
\newtheorem{lem}[thm]{Lemma}
\newtheorem{questionconjecture}[thm]{Question/Conjecture}

\newtheorem{problem}[thm]{Problem}

\theoremstyle{definition}

\numberwithin{equation}{section}

\newtheorem{rem}[thm]{Remark}
\newtheorem{ex}[thm]{Example}
\newtheorem{defn}[thm]{Definition}

\def\bbP{\mathbb{P}}
\def\bbQ{\mathbb{Q}}
\def\bbT{\mathbb{T}}

\def\bbZ{\mathbb{Z}}

\def\bfu{\mathbf{u}}
\def\bfv{\mathbf{v}}
\def\bfx{\mathbf{x}}
\def\bfy{\mathbf{y}}
\def\bfz{\mathbf{z}}
\def\bfZ{\mathbf{Z}}

\def\calF{\mathcal{F}}

\newcommand{\overunder}[2]{
\!\begin{array}{c}
\scriptstyle{#1}\\[-.1in]
-\!\!\!-\!\!\!-\\[-.1in]
\scriptstyle{#2}
\end{array}
\!
}

\begin{document}
\bibliographystyle{amsalpha}

\title[Synchronicity phenomenon
in cluster patterns
]
{
Synchronicity phenomenon \\
in cluster patterns
}

\address{\noindent Graduate School of Mathematics, Nagoya University, 
Chikusa-ku, Nagoya,
464-8604, Japan}
\email{nakanisi@math.nagoya-u.ac.jp}
\author{Tomoki Nakanishi}


\date{}
\maketitle
\begin{abstract}
It has been known that several objects  such as  cluster variables, coefficients, seeds, and $Y$-seeds  in different cluster patterns with common exchange matrices  share the same periodicity under mutations. We call it synchronicity phenomenon in cluster patterns.
 In this expository note we explain the mechanism of   synchronicity  based on several fundamental results on cluster algebra theory such as  separation formulas,  sign-coherence, Laurent positivity, duality, and detropicalization obtained by several authors.
 We also show that all synchronicity properties studied in this paper are naturally extended to cluster patterns of generalized cluster algebras, up to the Laurent positivity conjecture.

\end{abstract}


\section{Introduction: Synchronicity problem}
\label{sec:introduction}
Consider a pair $(\bfx,B)$,
where
$\bfx=(x_1,\dots, x_n)$ is an $n$-tuples of commuting formal variables,
and $B=(b_{ij})_{i,j=1}^n$ is a skew-symmetrizable integer matrix.
For any $k\in \{1,\dots, n\}$,
we obtain a new pair $(\bfx',B')$ by  the following transformation:
   \begin{align}
\label{eq:xmut0}
x'_i
&=
\begin{cases}
\displaystyle
x_k^{-1}\left( \prod_{j=1}^n x_j^{[-b_{jk}]_+}
\right)
( 1+\hat{y}_k)
& i=k,
\\
x_i
&i\neq k,
\end{cases}
\\
\label{eq:bmut0}
b'_{ij}&=
\begin{cases}
-b_{ij}
&
\text{$i=k$ or $j=k$,}
\\
b_{ij}+
b_{ik} [b_{kj}]_+
+
[-b_{ik}]_+b_{kj}
&
i,j\neq k,
\end{cases}
\end{align}
where
 \begin{align}
 [a]_+:=\max(a,0),
 \\
\label{eq:yhat0}
\hat{y}_i
:=\prod_{j=1}^n x_j^{b_{ji}}.
\end{align}
This is called the {\em mutation of a seed without coefficients at direction $k$}
\cite{Fomin02}.

Similarly, 
consider a pair $(\bfy,B)$,
where
$\bfy=(y_1,\dots, y_n)$ is another $n$-tuples of commuting formal variables,
and $B=(b_{ij})_{i,j=1}^n$ is a skew-symmetrizable integer matrix.
For any $k\in \{1,\dots, n\}$,
we obtain a new pair $(\bfy',B')$, where $B'$ is given by \eqref{eq:bmut0}
while $\bfy'$ is given by
\begin{align}
\label{eq:ymut0}
y'_i
&=
\begin{cases}
\displaystyle
y_k^{-1}
& i=k,
\\
y_i y_k^{[b_{ki}]_+} (1+ y_k)^{-b_{ki}}
&i\neq k.
\end{cases}
\end{align}
This is called the {\em mutation of a $Y$-seed with coefficients
in a universal semifield at direction $k$} \cite{Fomin02,Fomin07}.

Now, starting with a common skew-symmetrizable matrix $B$,
let us  apply the {\em same sequence of mutations\/} at given directions
$k_1$, \dots, $k_L$ on $(\bfx, B)$ and $(\bfy,B)$ as follows:
\begin{align}
&(\bfx,B) \mathop{\rightarrow}^{k_1}
(\bfx',B') \mathop{\rightarrow}^{k_2}
\cdots
 \mathop{\rightarrow}^{k_L}
 (\tilde{\bfx}, \tilde{B}),
\\
&(\bfy,B) \mathop{\rightarrow}^{k_1}
(\bfy',B') \mathop{\rightarrow}^{k_2}
\cdots
 \mathop{\rightarrow}^{k_L}
 (\tilde{\bfy}, \tilde{B}).
\end{align}
Then, though it is rare, it may happen that  $\tilde{\bfx}=\bfx$ (resp., $\tilde{\bfy}=\bfy$).
Let us call such a sequence of mutations as
 a {\em period of $x$-variables (resp. $y$-variables)}, respectively.
 
The following question is frequently asked:
 
 \begin{questionconjecture}[e.g., {\cite[Footnote 3 in Section 1.2]{Fock03}}]
 \label{question:xy1}
  Do the periodicities of $x$- and $y$-variables always
 coincide? In other words, does the following
 ``synchronicity"  of $x$- and $y$-variables
 hold?
  \begin{align}
  \label{eq:xy1}
  \tilde{\bfx}=\bfx 
  \quad
  \Longleftrightarrow
  \quad
  \tilde{\bfy}=\bfy.
 \end{align}
 \end{questionconjecture}
 
 This question, or the conjecture that the synchronicity
  \eqref{eq:xy1} 
holds, naturally arose from the observation
 that it indeed holds in
  several examples, such as the  rank 2 cluster algebras
  \cite{Fomin02},
  the finite type cluster algebras
  \cite{Fomin03a,Fomin03b},
 the  $T$-systems and the $Y$-systems of several types \cite{Inoue10c},
 etc.
 
 Two transformations \eqref{eq:xmut0} and \eqref{eq:ymut0}
 certainly  look closely related.
To be more specific, it has been well known  
(e.g.,  \cite[Proposition 3.9]{Fomin07})
that 
under the mutation \eqref{eq:xmut0}
 the
 variables in  \eqref{eq:yhat0} ({\em $\hat{y}$-variables})
 also mutate   just like $y$-variables, namely,
 \begin{align}
\label{eq:yhatmut0}
\hat{y}'_i
&=
\begin{cases}
\displaystyle
\hat{y}_k^{-1}
& i=k,
\\
\hat{y}_i \hat{y}_k^{[b_{ki}]_+} (1+ \hat{y}_k)^{-b_{ki}}
&i\neq k.
\end{cases}
\end{align}
If $B$ is nondegenerate, the initial $\hat{y}$-variables are algebraically independent;
therefore, the implication 
$  \tilde{\bfx}=\bfx 
  \
  \Longrightarrow
  \
  \tilde{\bfy}=\bfy$ holds
 by \eqref{eq:yhatmut0}.
Furthermore, if $\det B = \pm 1$,
one can invert the relation \eqref{eq:yhat0},
and the opposite implication also holds.
Therefore, the synchronicity \eqref{eq:xy1} holds if $\det B = \pm 1$.
(We may prove the opposite implication for a more general nondegenerate  $B$
by properly extending the field where $y$-variables live, but we do not touch this issue here.)

In many important and interesting cases,
 the matrix $B$ is degenerate.
Then,
in general, we do not yet know  how to prove the synchronicity \eqref{eq:xy1} by directly working on
   \eqref{eq:xmut0} and \eqref{eq:ymut0}.
    However, thanks to the recent development of cluster algebra theory,
  one can  prove it as a consequence of a collection of  known fundamental results in cluster algebra theory by several authors. In other words,  the cluster algebra structure (to be more precisely, the {\em cluster pattern structure}) behind it is essential  in our derivation.
 
Our proof of \eqref{eq:xy1} based on the known results is relatively elementary,
 in the sense that one can work in the framework of \cite{Fomin07}
 without relying on some other machinery.
 Therefore,
 we believe that   the fact  \eqref{eq:xy1}  and its proof   have been already known 
 for some time among  experts including the author.
In fact, a proof of the implication  from  left to right in \eqref{eq:xy1}
 was already presented in  \cite[Proposition 6.1]{Cao18}.
  However, the fact \eqref{eq:xy1}  seems not to be well known,
  and it is still a frequently asked question
  by colleagues to the author.  So in this note we present a proof of \eqref{eq:xy1},  together with  other  related synchronicity properties in cluster patterns
  scattering around in the literature as theorems and conjectures.
We note that our  proof  of the implication  from  left to right in \eqref{eq:xy1}
is different from  the one in  \cite{Cao18}.

The paper  consists of two parts:
Part \ref{part:CA} (Sections \ref{sec:cluster}--\ref{sec:open}) is for the results on cluster patterns in (ordinary) cluster algebras,
and Part \ref{part:GCA}  (Sections \ref{sec:cluster2}--\ref{sec:synchronicity2}) is for the results on cluster patterns in {\em generalized cluster algebras}
introduced by \cite{Chekhov11}.
In Sections \ref{sec:cluster}--\ref{sec:further} we give a concise review of fundamental properties on
cluster patterns that we need.
Section \ref{sec:synchronicity} is the main section of Part \ref{part:CA},
  where we derive and present several   synchronicity
 properties based on the results in the previous sections.
Overall we pay special attention on  presenting how these results are logically related.
See the introduction of Section \ref{sec:further} for more about this point.
  We conclude Part \ref{part:CA} by addressing an important open problem in Section \ref{sec:open}.
In Part \ref{part:GCA}  (Sections \ref{sec:cluster2}--\ref{sec:synchronicity2})
 we  show that all synchronicity properties 
 studied in Part \ref{part:CA} are naturally extended
 to cluster patterns of generalized cluster algebras,
  up to the Laurent positivity conjecture.
  Moreover,
we show that any  cluster pattern of GCA synchronizes with its ``companion cluster patterns" of CA.

Part \ref{part:CA}  is regarded as an extended note of the talks given at the workshops ``Cluster Algebras: Twenty Years On" at CIRM, Luminy, 2018, March and ``School on Cluster Algebras"
at ICTS, Bangalore, 2018, December.

We thank Sergey Fomin for
encouragement to write up this note
and also for
useful comments and suggestions on the manuscript.
This work was partially supported by JSPS KAKENHI Grant Number 16H03922.

\newpage
\part[I]{Synchronicity in cluster patterns of  cluster algebras}
\label{part:CA}

\section{Cluster patterns}
\label{sec:cluster}

Let us recall some basic notions in cluster algebra theory,
mostly following \cite{Fomin07}.
Consult \cite{Fomin07} for further details.
Throughout the paper we fix a positive integer $n$.

A square integer matrix $B=(b_{ij})_{i,j=1}^n$
is said to be {\em skew-symmetrizable\/}
if there is a diagonal  matrix $D=\mathrm{diag}(d_1,\dots,d_n)$
with positive integer  diagonal entries $d_1,
\dots, d_n$ such that $DB$ is skew-symmetric,
i.e.,
$
d_i b_{ij}= - d_jb_{ji}
$
holds. Such a matrix $D$ is called a {\em (left) skew-symmetrizer\/} of $B$.

 A  {\em semifield} $\mathbb{P}$  is 
a multiplicative abelian group 
equipped with an ``addition" $\oplus$ which is commutative, associative,
and distributive with respect to the multiplication.
Then, the group ring $\bbZ\bbP$ of $\mathbb{P}$
is a domain.
Let $\bbQ\bbP$ be the fraction field of $\bbZ\bbP$.

The following three examples of semifields are especially important in cluster algebra theory.

\begin{ex}
\label{ex:sf1}
(1). ({\em Universal semifield $\mathbb{Q}_{\mathrm{sf}}(\bfu)$
of $\bfu$}.)
Let $\bfu=(u_1,\dots, u_m)$ be an $m$-tuple of 
commuting formal variables.
We say that a 
rational function $f(\bfu)\in \bbQ(\bfu)$ in $\bfu$
has a {\em subtraction-free expression\/}
if it is expressed as
$f(\bfu)=p(\bfu)/q(\bfu)$, where $p(\bfu)$ and $q(\bfu)$
are nonzero polynomials in $u$ whose coefficients
are {\em positive\/} integers.
Let   $\mathbb{Q}_{\mathrm{sf}}(\bfu)$
be the set of all rational functions in $\bfu$
which have subtraction-free expressions.
Then, $\mathbb{Q}_{\mathrm{sf}}(\bfu)$
is a semifield by the usual 
multiplication and addition in $\bbQ(\bfu)$.
\par
(2). ({\em Tropical semifield  $\mathrm{Trop}(\bfu)$ of $\bfu$}.)
Let $\bfu=(u_1,\dots, u_m)$ be an $m$-tuple of 
commuting formal variables.
Let $\mathrm{Trop}(\bfu)$ be the set of all
Laurent monomials of $\bfu$ with coefficient 1,
which is a multiplicative abelian group by the usual
multiplication.
We define the addition $\oplus$ by
\begin{align}
\label{eq:ts1}
\prod_{i=1}^m u_i^{a_i} \oplus
\prod_{i=1}^m u_i^{b_i}
:=
\prod_{i=1}^m u_i^{\min(a_i,b_i)} .
\end{align}
Then,  $\mathrm{Trop}(\bfu)$ becomes a semifield.
The addition $\oplus$ is called the {\em tropical sum}.

\par
(3). ({\em Trivial semifield\/} $\mathbf{1}$.)
Let $\mathbf{1}=\{1\}$ be the trivial multiplicative group.
We define the addition by
$1\oplus 1 =1$.
Then, $\mathbf{1}$ becomes a semifield.
\end{ex}
 
\begin{defn}[Seeds]
\label{defn:seed1}
Let $\bbP$ be any semifield,
and let
$\calF$  be a field which is isomorphic to
the rational function field of $n$-variables
with coefficients in the field $\bbQ\bbP$.
A {\em (labeled) seed with coefficients in $\bbP$\/}
is a triplet
$\Sigma=(\bfx,\bfy,B)$,
where 
$\bfx=(x_1,\dots,x_n)$ 
is an
$n$-tuple of any algebraically independent   elements
(called  {\em cluster variables})
in $\calF$,
$\bfy=(y_1,\dots,y_n)$ is an
$n$-tuple of any  elements
 (called  {\em coefficients})
in $\bbP$,
and $B=(b_{ij})_{i,j=1}^n$
is a skew-symmetrizable integer matrix
 (called an {\em exchange matrix}).
\end{defn}

In this note let us  call $x_i$'s and $y_i$'s in the above as
{\em $x$-variables} and {\em $y$-variables}, respectively,
 in view of their equal roles in the periodicity phenomenon.
They are also called {\em cluster $\mathcal{A}$- and $\mathcal{X}$-coordinates}
in other fundamental references on cluster algebras by Fock and Goncharov \cite{Fock03,Fock07}.

\begin{defn}[Seed mutations]
\label{defn:mutation1}
For any seed $\Sigma=(\bfx,\bfy,B)$
and any $k\in \{1,\dots,n\}$,
we define a new seed 
$\Sigma'=(\bfx',\bfy',B')$
by the following formulas:
\begingroup
\allowdisplaybreaks
\begin{align}
\label{eq:xmut1}
x'_i
&=
\begin{cases}
\displaystyle
x_k^{-1}\left( \prod_{j=1}^n x_j^{[-b_{jk}]_+}
\right)
\frac{ 1+\hat{y}_k}{ 1\oplus y_k}
& i=k,
\\
x_i
&i\neq k,
\end{cases}
\\
\label{eq:ymut1}
y'_i
&=
\begin{cases}
\displaystyle
y_k^{-1}
& i=k,
\\
y_i y_k^{[b_{ki}]_+} (1\oplus y_k)^{-b_{ki}}
&i\neq k,
\end{cases}
\\
\label{eq:bmut1}
b'_{ij}&=
\begin{cases}
-b_{ij}
&
\text{$i=k$ or $j=k$,}
\\
b_{ij}+
b_{ik} [b_{kj}]_+
+
[-b_{ik}]_+b_{kj}
&
i,j\neq k,
\end{cases}
\end{align}
where $\hat{y}_k$ in \eqref{eq:xmut1}
is defined by
 \begin{align}
\label{eq:yhat1}
\hat{y}_i
:=y_i \prod_{j=1}^n x_j^{b_{ji}}
\in \mathcal{F}.
\end{align}
\endgroup
 The seed $\Sigma'$ is called the {\em mutation of $\Sigma$ at direction $k$},
and denoted by $\mu_k(\Sigma)$. 
Note that a skew-symmetrizer $D$ of $B$ is also a skew-symmetrizer of $B'$.
\end{defn}

The mutations are involutive, i.e., $\mu_k (\Sigma')=\Sigma$ in the above.

Let $\bbT_n$ be the $n$-regular  tree graph
where the edges are labeled by $1$,\dots, $n$
such that the $n$ edges attached to each vertex 
have different labels.
By abusing the notation,
the set of vertices of $\bbT_n$ is also denoted by 
 $\bbT_n$.

\begin{defn}[Cluster patterns]
\label{defn:clusterpattern1}
A family 
of seeds 
$\mathbf{\Sigma}=\{ \Sigma_t
=(\bfx_t,\bfy_t,B_t)
\mid t\in \bbT_n 
\}$ with coefficients in a semifield $\bbP$  indexed by $\bbT_n$
is called a {\em cluster pattern with coefficients in $\bbP$\/} if,
for any vertices $t,t'\in \bbT_n$ connected by an edge
labeled by $k$,
the equality $\Sigma_{t'}=\mu_k( \Sigma_t)$ holds.
\end{defn}

Following \cite{Fomin07}, for  a seed $\Sigma_t=(\bfx_t,\bfy_t,B_t)$ at $t$
in a cluster pattern,
we use the notation
\begin{align}
\label{eq:xybt1}
\bfx_t=(x_{1;t},\dots,x_{n;t}),
\quad
\bfy_t=(y_{1;t},\dots,y_{n;t}),
\quad
B_t=(b_{ij}^t)_{i,j=1}^n.
\end{align}

In parallel to Example \ref{ex:sf1} we introduce three families of cluster patterns.

\begin{ex}
\label{ex:cp1}
(1). ({\em Cluster pattern with universal coefficients})
This is a cluster pattern $\mathbf{\Sigma}=\{ \Sigma_t
=(\bfx_t,\bfy_t,B_t)
\mid t\in \bbT_n 
\}$  with coefficients in $\mathbb{P}$
where there is some $t_0\in \mathbb{T}_n$
such that, for the universal semifield  $\mathbb{Q}_{\mathrm{sf}}(\bfu)$
of a given $n$-tuple of variables $\bfu=(u_1,\dots, u_n)$,
the semifield homomorphism
$\varphi_{t_0}:\mathbb{Q}_{\mathrm{sf}}(\bfu)\rightarrow \mathbb{P}$ defined by
$u_i \mapsto y_{i;t_0}$ is injective.
In this case one can safely set $\mathbb{P}=\mathbb{Q}_{\mathrm{sf}}(\bfy_{t_0})$
from the beginning.
Moreover, one can replace it with $\mathbb{P}=\mathbb{Q}_{\mathrm{sf}}(\bfy_{t})$
for any other  $t\in \mathbb{T}_n$ if necessary,
 thanks to the involution property of the mutations.
 Sometimes we  call $\bfy_t$'s the {\em universal $y$-variables.}

We remark that no special name is given for such a cluster pattern in \cite{Fomin07},
and the terminology ``universal coefficients" is used for a different concept
\cite[Definition 12.3]{Fomin07} therein.
\par
(2). 
({\em Cluster pattern with principal coefficients at $t_0$.})
This is a cluster pattern $\tilde{\mathbf{\Sigma}}[{t}_0]=\{ \tilde{\Sigma}_t
=(\tilde{\bfx}_t,\tilde{\bfy}_t,B_t)
\mid t\in \bbT_n 
\}$  with coefficients in $\mathbb{P}$
where there is some ${t}_0\in \mathbb{T}_n$
such that, for the tropical semifield  $\mathrm{Trop}(\bfu)$
of a given $n$-tuple of variables $\bfu=(u_1,\dots, u_n)$,
there is an injective semifield homomorphism
$\varphi_{{t}_0}:\mathrm{Trop}(\bfu)\rightarrow \mathbb{P}$ such that
$u_i \mapsto \tilde{y}_{i;{t}_0}$.
In this case one can safely set $\mathbb{P}=\mathrm{Trop}(\tilde{\bfy}_{{t}_0})$
from the beginning.
In contrast to the previous case
one cannot replace it with $\mathrm{Trop}(\tilde{\bfy}_{t})$ with other 
$t\in \mathbb{T}_n$, in general, because 
the mutation \eqref{eq:ymut1} is {\em  incompatible} with the tropical sums
of $\mathrm{Trop}(\tilde{\bfy}_{t})$ and $\mathrm{Trop}(\tilde{\bfy}_{t'})$.
Sometimes we  call $\tilde{\bfy}_t$'s the {\em tropical $y$-variables} with the tropicalization point ${t}_0$.

\par
(3).
({\em Cluster pattern without  coefficients.})
This is a cluster pattern $\underline{\mathbf{\Sigma}}=\{ \underline{\Sigma}_t
=(\underline{\bfx}_t,\underline{\bfy}_t,B_t)
\mid t\in \bbT_n 
\}$ with coefficients in
a trivial semifield $\mathbf{1}$.
Since all $y$-variables are $1$,
one can safely omit them from  seeds; namely,
we set $\underline{\Sigma}_t
=(\underline{\bfx}_t,B_t)$.
\end{ex}

\section{Separation formulas}

We recall the most fundamental  theorem
by  \cite{Fomin07}
 on
the structure of seeds in cluster patterns.

Let  $\mathbf{\Sigma}=\{ \Sigma_t
=(\bfx_t,\bfy_t,B_t)
\mid t\in \bbT_n 
\}$ be a cluster pattern
with coefficients in any semifield $\bbP$.
We arbitrary choose a distinguished
point (the {\em initial point\/}) $t_0$ in $\bbT_n$.
The seed $(\bfx_{t_0},\bfy_{t_0},B_{t_0})$ at $t_0$ is called
the {\em initial seed} of $\mathbf{\Sigma}$.
We use the following simplified notation for them:
\begin{align}
\label{eq:initial2}
\bfx_{t_0}=\bfx=(x_{1},\dots,x_{n}),
\quad
\bfy_{t_0}=\bfy=(y_{1},\dots,y_{n}),
\quad
B_{t_0}=B=(b_{ij})_{i,j=1}^n.
\end{align}

Let us extract the family of the exchange matrices 
$\mathbf{B}=\{ B_t
\mid t\in \bbT_n 
\}$
from the cluster pattern $\mathbf{\Sigma}$,
and  call it a {\em $B$-pattern}.
Following \cite{Fomin07},
we introduce
a family of  quadruplets $\mathbf{\Gamma}(\mathbf{B},t_0)
=
\{ \Gamma_t=(\mathbf{F}_{t}(\bfu), G_t, C_t, B_t)
\mid t\in \mathbb{T}_n \}$
uniquely determined 
from $\mathbf{B}$ and  any initial point $t_0$
as follows.
For each $t\in \mathbb{T}_n$,
$\Gamma_t=(\mathbf{F}_{t}(\bfu), G_t, C_t, B_t)$
is a quadruplet (let us call it an {\em $FGC$-seed}),
where
\begin{itemize}
\item
 $\mathbf{F}_{t}(\bfu)=(F_{1;t}(\bfu),\dots, F_{n;t}(\bfu))$ is an $n$-tuple
of  rational functions (called  {\em $F$-polynomials}) in
$n$-tuple of variables $\bfu=(u_1,\dots,u_n)$
with coefficients in $\bbQ$
having a subtraction-free expression,
i.e., $F_{i;t}(\bfu)\in \bbQ_{\mathrm{sf}}(\bfu)$,
\item
and $G_t$ and $C_t$ are $n\times n$ integer matrices
(called a {\em $G$-matrix} and a {\em $C$-matrix}).
\end{itemize}
At the initial point $t_0$, they are given by
\begin{align}
\label{eq:initial1}
{F}_{i;t_0}(\bfu)=1 \quad (i=1,\dots,n),
\quad
G_{t_0}=C_{t_0}=I,
\end{align}
where $I$ is the identity matrix.
For $t,t'\in \mathbb{T}_n$
which  are connected by an edge labeled by $k$,
$\Gamma_t$ and $\Gamma_{t'}$ are related by the following {\em mutation
at $k$}:
\begingroup
\allowdisplaybreaks
\begin{align}
\label{eq:Fmut1}
 F_{i;t'}(\bfu)&=
 \begin{cases}
\frac
 {
  \displaystyle
 M_{k;t}(\bfu)
 }
{ \displaystyle
 F_{k;t}(\bfu)
 }
   &
 i= k
 \\
 F_{i;t}(\bfu)  &
 i \neq k,
 \end{cases}
 \\
 \label{eq:gmut1}
 g_{ij}^{t'}&=
 \begin{cases}
 \displaystyle
 -g_{ik}^t
 + \sum_{\ell=1}^n g_{i\ell}^t [-b_{\ell k}^t]_+
 -  \sum_{\ell=1}^nb_{i\ell}  [-c_{\ell k}^t]_+ 
 &
 j= k
 \\
 g_{ij}^t  &
 j \neq k,
 \end{cases}
 \\
 \label{eq:cmut1}
 c_{ij}^{t'}&=
 \begin{cases}
 -c_{ik}^t
 &
 j= k,
 \\
 c_{ij}^t + c_{ik}^t [b_{kj}^t]_+
 + [-c_{ik}^t]_+ b_{kj}^t
 &
 j \neq k,
 \end{cases}
\end{align}
\endgroup
where
\begin{align}
\begin{split}
\label{eq:M1}
 M_{k;t}(\bfu)
 &=
    \prod_{j=1}^{n}
  u_j^{[c_{jk}^t]_+}
  F_{j;t}(\bfu)^{[b_{jk}^t]_+}
+
    \prod_{j=1}^{n}
  u_j^{[-c_{jk}^t]_+}
  F_{j;t}(\bfu)^{[-b_{jk}^t]_+}
  \\
  &=
  \left(
      \prod_{j=1}^{n}
  u_j^{[-c_{jk}^t]_+}
  F_{j;t}(\bfu)^{[-b_{jk}^t]_+}
  \right)
  \left(
  1+
    \prod_{j=1}^{n}
    u_j^{c_{jk}^t}
  F_{j;t}(\bfu)^{b_{jk}^t}
  \right)
  .
  \end{split}
  \end{align}
Note that the transformation \eqref{eq:gmut1} involves the matrix
$B=B_{t_0}$ which does not depend on $t$.
Again, these mutations are involutive.
We call $\mathbf{\Gamma}(\mathbf{B},t_0)$ the {\em $FGC$-pattern
of $\mathbf{B}$ with the initial point $t_0$}.

As the name suggests,
a seemingly rational function $F_{i;t}(\bfu)\in \bbQ_{\mathrm{sf}}(\bfu)$
is indeed a polynomial in $\bfu$. 
(This fact follows from the celebrated {\em Laurent 
phenomenon\/} \cite{Fomin02,Fomin03a}.)

\begin{thm}[{\cite[Proposition 3.6]{Fomin07}}]
\label{thm:Fpoly1}
For any $i=1,\dots,n$ and $t\in \bbT_n$,
the function $F_{i;t}(\bfu)$ is a polynomial in $\bfu$
with coefficients in $\bbZ$.
\end{thm}

We emphasize that, in our formulation, the initial point $t_0$ for a $FGC$-pattern
can be chosen independently to the tropicalization point $t_0$ of
any cluster pattern with principal coefficients in Example \ref{ex:cp1} (2)
that we will consider in this paper.
However, when they coincide,
the $C$-matrices   are naturally  identified as 
the (exponents of) tropical $y$-variables as follows:

\begin{prop}[{\cite[Remark 3.2]{Fomin07}}]
\label{prop:trop1}
Let $\tilde{\mathbf{\Sigma}}[t_0]=\{ \tilde{\Sigma}_t
=(\tilde{\bfx}_t,\tilde{\bfy}_t,B_t)
\mid t\in \bbT_n 
\}$
be a cluster pattern with principal coefficients at  any point $t_0$,
and 
let
$\mathbf{\Gamma}(\mathbf{B},t_0)
=
\{ \Gamma_t=(\mathbf{F}_{t}(\bfu), G_t, C_t, B_t)
\mid t\in \mathbb{T}_n \}$
be the $FGC$-pattern  of $\mathbf{B}$ with
 the initial point $t_0$.
 Then, the following formula holds for any $t\in \mathbb{T}_n$:
\begin{align}
\label{eq:pr1}
\tilde{y}_{i;t}&=\prod_{j=1}^n \tilde{y}_j ^{c^t_{ji}}.
\end{align}
\end{prop}

Now we present formulas expressing
 arbitrary $x$- and $y$-variables
$\bfx_t$, $\bfy_t$
in terms of the initial $x$- and $y$-variables
$\bfx$, $\bfy$
together with
$G$- and $C$-matrices and $F$-polynomials
with any initial point $t_0$.
This is one of the most fundamental 
properties of cluster patterns,
which makes the cluster pattern (or cluster algebra) structure so useful
in various applications.

\begin{thm}[{Separation Formulas \cite[Proposition 3.13. Corollary 6.13]{Fomin07}}]
\label{thm:sep1}
For  a cluster pattern
 $\mathbf{\Sigma}=\{ \Sigma_t
=(\bfx_t,\bfy_t,B_t)
\mid t\in \bbT_n 
\}$
with coefficients in any semifield $\bbP$,
the following formulas hold
for any $t\in \mathbb{T}_n$:
\begin{align}
\label{eq:sep1}
x_{i;t}&=
\left(
\prod_{j=1}^n
x_j^{g_{ji}^t}
\right)
\frac{F_{i;t}(\hat{\bfy})}{F_{i;t}\vert_{\bbP}(\bfy)},
\quad
\hat{y}_i=y_i \prod_{j=1}^n x_j^{b_{ji}}
\in \mathcal{F},
\\
\label{eq:sep2}
y_{i;t}&=
\left(
\prod_{j=1}^n
y_j^{c_{ji}^t}
\right)
\prod_{j=1}^n
F_{j;t}\vert_{\bbP}(\bfy)^{b_{ji}^t},
\end{align}
where $F_{i;t}\vert_{\bbP}(\bfy)$ is the evaluation
of $F_{i,t}(\bfu)$ in $\mathbb{P}$ at $\bfy$.
\end{thm}

\begin{rem}
Originally in \cite{Fomin07}, $FGC$-patterns are defined through cluster patterns
with principal coefficients, and the proof of
Theorem \ref{thm:sep1} also relied on it.
However, it is  straightforward to prove Theorem \ref{thm:sep1}
only from \eqref{eq:initial1}--\eqref{eq:M1} by induction on $t$
(since we already  know this result!).
Moreover, the original definition of   $FGC$-patterns in \cite{Fomin07}
can be easily recovered from Theorem \ref{thm:sep1}.
We leave it as an exercise to the readers.
\end{rem}

\begin{rem}
In \cite{Fomin07} only the first formula \eqref{eq:sep1} is called the separation formula,
meaning that it separates the additions of $\mathcal{F}$ and $\mathbb{P}$
in the numerator and the denominator therein.
In view of  the parallelism of $x$- and $y$-variables, we also call 
the second one \eqref{eq:sep2}  the separation formula.
Here the ``separation'' has an additional meaning  in view of 
Proposition \ref{prop:trop1} that
 the {\em tropical part} (involving $C$ and $G$ matrices) and {\em nontropical part}
 (involving $F$-polynomials) are separated therein. 
  \end{rem}
 
 Let $S_n$ be the symmetric group of degree $n$.
For  a seed $\Sigma_t=(\bfx_t,\bfy_t,B_t)$ at $t$
and a permutation $\sigma \in S_n$,
we define the left action of $\sigma$
as
$\sigma\Sigma_t=(\sigma \bfx_t,\sigma \bfy_t,\sigma B_t)$,
where
 \begin{align}
  \label{eq:s4}
 \sigma(x_{1;t},\dots,x_{n;t})&=
 (x_{\sigma^{-1}(1);t},\dots,x_{\sigma^{-1}(n);t}),
 \\
  \label{eq:s5}
 \sigma(y_{1;t},\dots,x_{n;t})&=
 (y_{\sigma^{-1}(1);t},\dots,y_{\sigma^{-1}(n);t}),
 \\
  \label{eq:s6}
  \sigma B_t &=B',\quad b'_{ij}=b^t_{\sigma^{-1}(i)\sigma^{-1}(j)}.
 \end{align}
  This action is compatible with mutations; namely, we have
 \begin{align}
     \label{eq:comp1}
   \mu_{\sigma(k)}(  \sigma \Sigma_t)&=\sigma \mu_k (\Sigma_t).
   \end{align}
 Similarly,
 for a 
 $FGC$-seed  $\Gamma_t=(\mathbf{F}_t(\bfu), G_t,C_t,B_t)$ at $t$
and a permutation $\sigma\in S_n$,
 we define the action of $\sigma$ as
$\sigma\Gamma_t=(\sigma \mathbf{F}_t(\bfu), \sigma G_t, \sigma C_t, \sigma B_t)$,
where
 \begin{align}
  \label{eq:s1}
 \sigma(F_{1;t}(\bfu), \dots, F_{n;t}(\bfu))&=(F_{\sigma^{-1}(1);t}(\bfu),\dots, 
 F_{\sigma^{-1}(n);t}(\bfu)),
 \\
 \label{eq:s2}
 \sigma G_t &=G',\quad g'_{ij}=g^t_{i\sigma^{-1}(j)},\\
 \label{eq:s3}
 \sigma C_t &=C',\quad c'_{ij}=c^t_{i\sigma^{-1}(j)}.
 \end{align}
Again, we have
   \begin{align}
      \label{eq:comp2}
   \mu_{\sigma(k)}(  \sigma \Gamma_t)&=\sigma \mu_k (\Gamma_t).
 \end{align}
(Caution: For the mutation in the left hand side 
of \eqref{eq:comp2} we keep $B$
in the last term of the first line of \eqref{eq:gmut1} as it is, 
not replacing it with $\sigma B$.
Otherwise, the equality \eqref{eq:comp2} fails.)
Observe that \eqref{eq:s1}--\eqref{eq:s3} and 
  \eqref{eq:s4}--\eqref{eq:s6} are compatible in view of the separation formulas
  \eqref{eq:sep1}--\eqref{eq:sep2}.
  
Let us introduce
 the permutation matrix $P_{\sigma}$ associated with $\sigma\in S_n$ as
\begin{align}
\label{eq:p1}
P_{\sigma}=(p_{ij}), \ p_{ij}=\delta_{i,\sigma^{-1}(j)}.
\end{align}
Then, \eqref{eq:s6}, \eqref{eq:s2}, and \eqref{eq:s3}
are written as
\begin{align}
\label{eq:p2}
\sigma C_t=C_t P_{\sigma},
\
\sigma G_t=G_t P_{\sigma},
\
\sigma B_t=P_{\sigma}^T B_t P_{\sigma},
\end{align}
where $M^T$ is the transposition of a matrix $M$.

\section{Further fundamental results}
\label{sec:further}

To extract the power of the separation formulas in our periodicity problem,
we need some additional properties on $G$- and $C$-matrices and $F$-polynomials.

It turns out that  all necessary results stem from  two fundamental and very deep results
in cluster algebra theory, namely, the {\em sign-coherence\/} and the {\em Laurent positivity},
 both of which were proved partially by several authors and 
 proved in full generality
  by the  recent seminal paper \cite{Gross14}.
Unfortunately, no elementary or direct proof is yet known in the framework of \cite{Fomin07},
and the known proofs (except for \cite{Lee15} notably) involve some construction or realization of cluster variables with relatively big machinery such as categorification
or scattering diagram method.

Our strategy is to accept this two fundamental results temporarily
as inputs from outside, and then work in the framework of \cite{Fomin07} in the rest,
hoping that we will have some elementary proofs in the framework of \cite{Fomin07} in some future.

Throughout this section we fix
the $FGC$-pattern
$\mathbf{\Gamma}(\mathbf{B},t_0)
=
\{ \Gamma_t=(\mathbf{F}_{t}(\bfu), 
\allowbreak G_t, C_t, B_t)
\mid t\in \mathbb{T}_n \}$
of any $B$-pattern $\mathbf{B}$ with any initial point $t_0$.
Also, any cluster patterns in this section  share the common 
$B$-pattern $\mathbf{B}$ with $\mathbf{\Gamma}(\mathbf{B},t_0)$.

\subsection{Sign-coherence}
\label{subsec:sign-coherence}
The first fundamental result  is the
following one.
\begin{thm}[{Sign-coherence, \cite[Corollary 5.5]{Gross14},
\cite[Theorem 1.7]{Derksen10}, \cite[Theorem 3.7 (1) and Remark 3.11]{Plamondon10b},
\cite[Theorem 8.8]{Nagao10}}]
\label{thm:sign}
Each $C$-matrix $C_t$ is column-sign coherent.
Namely, each column vector ($c$-vector) of $C_t$   is nonzero vector, and 
its components are either all nonnegative or all nonpositive.
\end{thm}

This was  conjectured by \cite[Conjecture 5.5]{Fomin07},
and proved for the skew-symmetric case by \cite{Derksen10,Plamondon10b,Nagao10}
and in general by \cite{Gross14}.

It was noticed by  \cite[Proposition 5.6]{Fomin07}
that
Theorem \ref{thm:sign} is 
 equivalent to the following simple statement 
 for $F$-polynomials.
 (The equivalence is easily proved from \eqref{eq:Fmut1} by  induction on $t$.)
 
 \begin{thm}
  \cite{Gross14,Derksen10, Plamondon10b, Nagao10}
 \label{thm:const1}
 The constant term of each $F$-polynomial $F_{i;t}(\bfu)$ is 1.
 \end{thm}
In fact, Theorem \ref{thm:sign} was proved  in this form in \cite{Derksen10, Nagao10}.

Let $D$ be a common skew-symmetrizer of (the matrices in) $\mathbf{B}$.

We have the following consequence of Theorem
\ref{thm:sign}. (It was proved earlier than Theorem \ref{thm:sign}
under the assumption of Theorem
\ref{thm:sign}.)

\begin{thm}[{\cite[Eqs.~(1,11), (1.15), (2.7), (2.9)]{Nakanishi11a}}]
\label{thm:dual1}
The following equalities hold for any $t\in \mathbb{T}_n$:
\begin{align}
\label{eq:dual1}
\text{(Duality)}
\quad
D^{-1} G_t^T D C_t &= I,
\\
\label{eq:bc1}
DB_t  &= C_t^T DB C_t,
\\
\label{eq:det1}
|\det G_t|& =|\det C_t|=1.
\end{align}
\end{thm}

 We say that a permutation $\sigma\in S_n$ is {\em compatible with $D=\mathrm{diag}(d_1,\dots,d_n)$}
 if 
 \begin{align}
 \label{eq:comd1}
 d_{\sigma(i)}=d_i 
 \quad (i=1,\dots n).
 \end{align}
With the permutation matrix $P_{\sigma}$ in \eqref{eq:p1},
the  condition \eqref{eq:comd1} is equivalently expressed as
\begin{align}
\label{eq:dp1}
DP_{\sigma}=P_{\sigma}D.
\end{align}

\begin{prop}
\label{prop:compat1}
The following properties hold.
\par
(1). If $C_{t_1} = \sigma C_{t_2}$ occurs for some $t_1,t_2\in \mathbb{T}_n$
and a permutation $\sigma\in S_n$,
then $\sigma$ is compatible with $D$.
\par
(2). If $G_{t_1} = \sigma G_{t_2}$ occurs for some $t_1,t_2\in \mathbb{T}_n$
and a permutation $\sigma\in S_n$,
then $\sigma$ is compatible with $D$.
\end{prop}
\begin{proof}
(1).
By \eqref{eq:dual1}, we have $DC_{t_1}D^{-1}=(G^T_{t_1})^{-1}$ and 
$DC_{t_2}^{-1}D^{-1}=G^T_{t_2}$.
By assumption, $C_{t_2}^{-1}C_{t_1}=P_{\sigma}$.
Therefore, $DP_{\sigma}D^{-1}=G^T_{t_2}(G^T_{t_1})^{-1}$.
By \eqref{eq:det1}, the righthand side is an integer matrix.
Therefore, $d_i  d_{\sigma(i)}^{-1}$ is an integer
for any $i=1,\dots,n$. 
In particular, $d_{\sigma(i)}\leq d_i$ for any $i$.
This is possible only  if $d_{\sigma(i)}=d_i$ for any $i$.
Therefore, $\sigma$ is compatible with $D$.
The proof of (2) is similar.
\end{proof}
We have the following corollary of Theorem \ref{thm:dual1}
and Proposition \ref{prop:compat1}.

\begin{cor} 
\label{cor:sigma1}
The following statements hold for $t_1,t_2\in \mathbb{T}_n$
and  a  permutation $\sigma\in S_n$,
where $P_{\sigma}$ is the permutation matrix in \eqref{eq:p1}.
\begin{gather}
\label{eq:dual3}
C_{t_1} = P_{\sigma}  \quad \Longleftrightarrow\quad  G_{t_1} =P_{\sigma}
\quad \Longrightarrow
\quad
 B_{t_1} = \sigma B,
\\
\label{eq:dual2}
C_{t_1} = \sigma C_{t_2}\quad \Longleftrightarrow\quad  G_{t_1} = \sigma G_{t_2}
 \quad\Longrightarrow  \quad B_{t_1} = \sigma B_{t_2}.
 \end{gather}
\end{cor}
\begin{proof}
They follow from
\eqref{eq:p2},
\eqref{eq:dual1},
\eqref{eq:bc1},
and  \eqref{eq:dp1},
thanks to the compatibility of $\sigma$ with $D$
in Proposition \ref{prop:compat1}.
\end{proof}

\subsection{Laurent positivity}
\label{subsec:Laurent}

The second fundamental result  is the
following one.
\begin{thm}[{Laurent positivity, \cite[Corollary 0.3]{Gross14},
\cite[Theorem 1.1]{Musiker09},
\cite[Corollary 3.39]{Kimura12},
\cite[Theorem 1.1]{Lee15},
\cite[Theorem 2.4]{Davison16}}]
\label{thm:positivity}
Each $F$-polynomial $F_{i;t}(\bfu)$ has
only positive coefficients.
\end{thm}

This was conjectured by \cite[Section 3]{Fomin07},
and proved for surface type by
\cite{Musiker09},
for acyclic case by \cite{Kimura12},
for  the skew-symmetric case by \cite{Lee15,Davison16},
and in general  by \cite{Gross14}.

Let us give a simple and useful criterion of the triviality of an $F$-polynomial,
which quickly follows
from Theorems \ref{thm:const1} and \ref{thm:positivity}.

\begin{lem}
\label{lem:Fcri1}
Let $\underline{\mathbf{\Sigma}}=\{ \underline{\Sigma}_t
=(\underline{\bfx}_t,B_t)
\mid t\in \bbT_n 
\}$
be a cluster pattern without coefficients,
and let $\underline{\hat{y}}_{i;t}$ be the one \eqref{eq:yhat1} for $\underline{\mathbf{\Sigma}}$, namely,
\begin{align}
\label{eq:yhat3}
\underline{\hat{y}}_{i;t}=\prod_{j=1}^n\underline{x}_{j:t}^{b^t_{ji}}.
\end{align}
Then, under the specialization $\underline{x}_{i:t_1}=1$ ($i=1,\dots,n$) at any point $t_1$,
we have the following inequality for any  $t,t'\in \mathbb{T}_n$ and $i=1,\dots,n$:
\begin{align}
F_{i;t}(\underline{\hat{\mathbf{y}}}_{t'})\vert_{\underline{x}_{1:t_1}=\dots
=\underline{x}_{n:t_1}=1}\geq 1.
\end{align}
Moreover, the equality holds if and only if 
$F_{i;t}({\mathbf{u}})=1$.
\end{lem}
\begin{proof}  By $\eqref{eq:xmut1}$, we have $\underline{x}_{i;t}>0$
 for any $t$ and $i$
under the specialization.
Therefore, by \eqref{eq:yhat3}, we have $\underline{\hat{y}}_{i;t}>0$ 
  for any $t$ and $i$.
Then, the claim is an immediate consequence of Theorems \ref{thm:const1} and \ref{thm:positivity}.
\end{proof}

The following theorem is a  very important   consequence of  
Theorems \ref{thm:sign}, \ref{thm:const1}, and \ref{thm:positivity} together,
where only the case $\sigma=\mathrm{id}$ was treated in
 \cite{Cao17}.
 It means that the tropical part ($G$-matrix) uniquely determines
 the nontropical part ($F$-polynomials).

\begin{thm}[{Detropicalization, \cite[Lemma 2.4 \& Theorem 2.5]{Cao17}}]
\label{thm:C-F}
The following statements hold for $t_1,t_2\in \mathbb{T}_n$
and  a  permutation $\sigma\in S_n$,
where $P_{\sigma}$ is the permutation matrix in \eqref{eq:p1}:
\begin{align}
\label{eq:gf1}
G_{t_1} = P_{\sigma} \quad &\Longrightarrow\quad 
F_{i;t_1}(\bfu)=1 \quad (i=1,\dots,n),
\\
\label{eq:gf2}
G_{t_1} = \sigma G_{t_2}\quad &\Longrightarrow\quad 
\mathbf{F}_{t_1}(\bfu)
=\sigma \mathbf{F}_{t_2}(\bfu).
\end{align}
\end{thm}

For completeness, we  present a proof for general $\sigma$,
slightly (and carefully) modifying the
proof of   \cite{Cao17}.
\begin{proof}
Proof of \eqref{eq:gf1}.
We prove it by
following the proof of Lemma 2.4 of  \cite{Cao17}.
\par
{\em Step 1.}
Suppose that $G_{t_1} =  P_{\sigma}$ for some $t_1$.
Then, by Corollary \ref{cor:sigma1}, we 
have $C_{t_1}=P_{\sigma}$.
Let $\tilde{\mathbf{\Sigma}}[t_0]=\{ \tilde{\Sigma}_t
=(\tilde{\bfx}_t,\tilde{\bfy}_t,B_t)
\mid t\in \bbT_n 
\}$
be a cluster pattern with principal coefficients at $t_0$.
Then, by Proposition \ref{prop:trop1}, we have
\begin{align}
\label{eq:yy1}
\tilde{y}_{i;t_1} = \tilde{y}_{\sigma^{-1}(i);t_0}\quad (i=1,\dots,n).
\end{align}
This implies that
the cluster pattern $\tilde{\mathbf{\Sigma}}$ is also viewed as
the one with principal coefficients at $t_1$.
Let 
$\mathbf{\Gamma}'(\mathbf{B},t_1)
=
\{ \Gamma'_t=(\mathbf{F}'_{t}(\bfu), G'_t, C'_t, B_t)
\mid t\in \mathbb{T}_n \}$
be the $FGC$-pattern  of $\mathbf{B}$ with the initial point $t_1$.
Then, by reversing the relation \eqref{eq:yy1}
and again by Proposition \ref{prop:trop1}, we have
$C'_{t_0}=P_{\sigma^{-1}}$. Therefore, we have
$G'_{t_0}=P_{\sigma^{-1}}$  by Corollary \ref{cor:sigma1}.

\par
{\em Step 2.}
Next, 
let $\underline{\mathbf{\Sigma}}=\{ \underline{\Sigma}_t
=(\underline{\bfx}_t,B_t)
\mid t\in \bbT_n 
\}$
be a cluster pattern without coefficients.
By applying the separation
formula \eqref{eq:sep1} with both initial points $t_0$ and $t_1$,
we have, for any $i=1,\dots,n$,
\begin{align}
\label{eq:xx1}
\underline{x}_{i:t_1}&=\underline{x}_{\sigma^{-1}(i):t_0}
F_{i;t_1}(\underline{\hat{\bfy}}_{t_{0}}),
\quad
\underline{\hat{y}}_{i;t_0}=\prod_{j=1}^n \underline{x}_{j;t_{0}}^{b^{t_0}_{ji}},
\\
\label{eq:xx2}
\underline{x}_{i:t_0}&=\underline{x}_{\sigma(i):t_1}
F'_{i;t_0}(\underline{\hat{\bfy}}_{t_1}),
\quad
\underline{\hat{y}}_{i;t_1}=\prod_{j=1}^n \underline{x}_{j;t_{1}}^{b^{t_1}_{ji}},
\end{align}
where we used $G_{t_1} =  P_{\sigma}$ and $G'_{t_0}=P_{\sigma^{-1}}$.
Then, by replacing $i$ in \eqref{eq:xx2} with $\sigma^{-1}(i)$
and multiplying it with \eqref{eq:xx1},
we obtain
\begin{align}
\label{eq:ff1}
1= F_{i;t_1}(\underline{\hat{\bfy}}_{t_{0}})
F'_{\sigma^{-1}(i);t_0}(\underline{\hat{\bfy}}_{t_1}).
\end{align}
It is important that the left hand side is $1$.
Now we do the  specialization $\underline{x}_{1;t_0}=\dots, =\underline{x}_{n;t_0}=1$
in \eqref{eq:ff1}.
Then, by Lemma \ref{lem:Fcri1},
we conclude that
 $F_{i;t_1}(\bfu)=F'_{\sigma^{-1}(i);t_0}(\bfu)=1$,
 which is the desired result.
 (Simplification of the proof using Lemma \ref{lem:Fcri1} is due to ourselves.)
 \par
Proof of \eqref{eq:gf2}.
Again, we prove it by
following the proof of Theorem  2.5 of  \cite{Cao17}.
Suppose that $G_{t_1}=\sigma G_{t_2}$ for some $t_1$ and $t_2$.
Then, by Corollary \ref{cor:sigma1}, we
have $C_{t_1}=\sigma C_{t_2}$ and $B_{t_1}=\sigma B_{t_2}$.
Suppose that  $t_2$ is connected to the initial point $t_0$ 
in $\mathbb{T}_n$
as
\begin{align}
t _0
\overunder{k_1}{} 
\cdots
\overunder{k_p}{} 
t_2.
\end{align}
Let $t_3\in \mathbb{T}_n$ be the one such that
\begin{align}
\label{eq:p3}
t_3
\overunder{\sigma(k_1)}{} 
\cdots
\overunder{\sigma(k_p)}{} 
t_1.
\end{align}
Then, by \eqref{eq:comp2} and ignoring $F$-polynomials therein, we have 
the following commutative diagram:
\begin{align}
\begin{matrix}
(G_{t_0},C_{t_0}, B_{t_0})=(I,I,B)
&
\displaystyle
\mathop{\rightarrow}^{k_1}
&
\cdots
&
\displaystyle
\mathop{\rightarrow}^{k_p}
&
(G_{t_2},C_{t_2}, B_{t_2})
\\
\sigma \downarrow
&&&&
\sigma \downarrow
\\
(G_{t_3},C_{t_3}, B_{t_3})\,
&
\displaystyle
\mathop{\rightarrow}^{\sigma(k_1)}
&
\cdots
&
\displaystyle
\mathop{\rightarrow}^{\sigma(k_p)}
&
(G_{t_1},C_{t_1}, B_{t_1}).
\end{matrix}
\end{align}
It follows that
\begin{align}
(G_{t_3},C_{t_3}, B_{t_3})=
(P_{\sigma}, P_{\sigma}, \sigma B).
\end{align}
Then, by \eqref{eq:gf1}, we have 
$F_{i,t_3}(\bfu)=1$ for any $i=1,\dots,n$.
Therefore, we have
\begin{align}
\Gamma_{t_3}=\sigma \Gamma_{t_0}.
\end{align}
Then, again by \eqref{eq:comp2},
we have the commutative diagram
\begin{align}
\begin{matrix}
\Gamma_{t_{0}}&
\displaystyle
\mathop{\rightarrow}^{k_1}
&
\cdots
&
\displaystyle
\mathop{\rightarrow}^{k_p}
&
\Gamma_{t_2}\,
\\
\sigma \downarrow \quad\
&&&&
\sigma \downarrow \quad\
\\
\Gamma_{t_3}
&
\displaystyle
\mathop{\rightarrow}^{\sigma(k_1)}
&
\cdots
&
\displaystyle
\mathop{\rightarrow}^{\sigma(k_p)}
&
\Gamma_{t_1}.
\end{matrix}
\end{align}
Therefore, we obtain 
\begin{align}
\Gamma_{t_1} = \sigma \Gamma_{t_2},
\end{align}
which proves the claim \eqref{eq:gf2}.
\end{proof}

\section{Synchronicity phenomenon}
\label{sec:synchronicity}

In this section we present several synchronicity properties
in cluster patterns
based on the results in the previous sections.
Many of them have already appeared partially and/or
in special cases such as skew-symmetric case, $\sigma=1$ case,
finite type case, geometric coefficients case, surface type, etc., in various literature
with several methods.
The references below are not complete by any means.
Also, many of the results may be derived  in a different manner
 via the scattering diagram method in \cite{Gross14}.
 However,
we believe that our approach is still useful
in view of cluster algebra theory
as explained in the introduction of Section \ref{sec:further}.

\subsection{$\sigma$-periodicity}
Let us formulate the notion of periodicity
for cluster patterns, including  the partial periodicity 
up to permutations $\sigma\in S_n$.

Let  $\mathbf{\Sigma}=\{ \Sigma_t
=(\bfx_t,\bfy_t,B_t)
\mid t\in \bbT_n 
\}$ be a cluster pattern
with coefficients in any semifield $\bbP$.
Suppose that $t_1,t_2 \in\mathbb{T}_n$
are connected in $\mathbb{T}_n$ as
\begin{align}
\label{eq:seq1}
t _1
\overunder{k_1}{} 
\cdots
\overunder{k_p}{} 
t_2,
\end{align}
so that
\begin{align}
\label{eq:seq2}
\Sigma_{t_2}=\mu_{k_p} \cdots \mu_{k_1}(\Sigma_{t_1}).
\end{align}

\begin{defn}[$\sigma$-periodicity]
We call a sequence of  mutations of seeds \eqref{eq:seq2}
 a 
{\em $\sigma$-period\/} if
\begin{align}
\Sigma_{t_1} = \sigma \Sigma_{t_2}.
\end{align}
In this case
we also say that seeds are {\em $\sigma$-periodic} under the sequence of mutations
\eqref{eq:seq2}.
Similarly, 
 we say that $x$-variables (resp., $y$-variables) are {\em $\sigma$-periodic} under the sequence of mutations \eqref{eq:seq2} if $\bfx_{t_1}=\sigma \bfx_{t_2}$ (resp.,
 $\bfy_{t_1}=\sigma \bfy_{t_2}$) holds.
\end{defn}

\subsection{$xy/GC$ synchronicity}
\label{subsec:xyGC}

In this subsection we  present some basic synchronicity property in cluster patterns,
which we call the {\em $xy/GC$ synchronicity}.
Namely, both periodicities of $x$- and $y$-variables coincide with 
the common periodicity of their tropical counterparts, $G$- and $C$-matrices
at any initial point $t_0$,
where some condition is assumed for $y$-variables.
It has been proved partially and/or
 in special cases, for example,  in \cite{Fomin07,Plamondon10b, Nagao10,Inoue10a, 
Cerulli12,Iwaki14a, Gross14, Cao17,Cao18}.
We regard it as the most basic synchronicity property,
because all other synchronicity properties studied in this paper are
obtained from it.

Let us divide
the statement of the $xy/GC$ synchronicity  into two theorems.
The first half of the statement is as follows:

\begin{thm}[$xy/GC$ synchronicity]
\label{thm:basic1}
Let
${{\mathbf{\Sigma}}}=\{ {\Sigma}_t
=({\bfx}_t,
{\bfy}_t,B_t)
\mid t\in \bbT_n 
\}$
be a cluster pattern with  coefficients in 
any  semifield $\mathbb{P}$,
and let
$\mathbf{\Gamma}(\mathbf{B},t_0)=
\{ \Gamma_t=(\mathbf{F}_{t}(\bfu), G_t, C_t, B_t)
\mid t\in \mathbb{T}_n \}$
be the $FGC$-pattern of $\mathbf{B}$ at any initial point $t_0$.
Then, 
for $t_1,t_2\in \mathbb{T}_n$
and  a  permutation $\sigma\in S_n$,
 the following three conditions are equivalent:
\par
(a). $G_{t_1}=\sigma G_{t_2}$.
\par
(b). $C_{t_1}=\sigma C_{t_2}$.
\par
(c). $\bfx_{t_1}=\sigma \bfx_{t_2}$.
\par
\noindent
Moreover, one of Conditions (a)--(c) implies the following condition:
\par
(d). $\bfy_{t_1}=\sigma \bfy_{t_2}$.
\end{thm}

\begin{proof} 
((a) $\Longleftrightarrow$ (b)).
This was already stated in
\eqref{eq:dual2}.
\par
 ((a) $\Longrightarrow$ (c), (d)).  Assume that 
 $G_{t_1}=\sigma G_{t_2}$.
 Then, by \eqref{eq:dual2} and  \eqref{eq:gf2}, we have $C_{t_1}=\sigma C_{t_2}$,
 $B_{t_1}=\sigma B_{t_2}$
 and  $F_{i;t_1}(u)=F_{\sigma^{-1}(i);t_2}(u)$.
Then, applying the separation formulas \eqref{eq:sep1} and \eqref{eq:sep2}
for $\bfx_{t_1}$  and $\bfy_{t_1}$, respectively, we obtain Conditions (c) and (d).
\par
((c) $\Longrightarrow$ (b)).
Assume that ${\bfx}_{t_1}=\sigma {\bfx}_{t_2}$.
Now let $\underline{\mathbf{\Sigma}}=\{ \underline{\Sigma}_t
=(\underline{\bfx}_t,B_t)
\mid t\in \bbT_n 
\}$ be a cluster pattern without coefficients.
First apply a semifield homomorphism (the trivialization map)
in the coefficients of $\bfx_{t}$
\begin{align}
\label{eq:semihom1}
\begin{matrix}
&{\mathbb{P}}
&\rightarrow &\mathbf{1} \ \\
 & {y}_{i;t}& \mapsto & 1,
\end{matrix}
\end{align}
then apply a homomorphism
 ${x}_{i;t} \mapsto  \underline{x}_{i;t}$  from the subfield of the ambient field
 $\mathcal{F}$ of $\mathbf{\Sigma}$ generated by
$\bfx_t$'s to 
the ambient field $\underline{\mathcal{F}}$ of $\underline{\mathbf{\Sigma}}$.
Then, we obtain  $\underline{\bfx}_{t_1}=\sigma \underline{\bfx}_{t_2}$.
\par
Let us take the $FGC$-pattern  $\mathbf{\Gamma}(\mathbf{B},t_2)=
\{ \Gamma'_t=(\mathbf{F'}_{t}(\bfu), G'_t, C'_t, B_t)
\mid t\in \mathbb{T}_n \}$
  {\em with the initial point $t_2$}.
  Applying the separation formula \eqref{eq:sep1}
for $\underline{\bf{x}}_{t_1}$ with the initial point $t_2$,
we have, for any $i=1,\dots,n$,
\begin{align}
\label{eq:xtilde1}
\underline{x}_{i;t_1}&=
\left(
\prod_{j=1}^n
\underline{x}_{j;t_2}^{g^{\prime t_1}_{ji}}
\right)
{F'_{i;t_1}(\underline{\hat{\bfy}}_{t_2})},
\quad
{\underline{\hat{y}}}_{i;t_2}= \prod_{j=1}^n \underline{x}_{j;t_2}^{b_{ji}^{t_2}}.
\end{align}
Since $\underline{\bfx}_{t_1}=\sigma \underline{\bfx}_{t_2}$,
we have
\begin{align}
\label{eq:xtilde2}
\underline{x}_{\sigma^{-1}(i);t_2}&=
\left(
\prod_{j=1}^n
\underline{x}_{j;t_2}^{g^{\prime t_1}_{ji}}
\right)
{F'_{i;t_1}(\underline{\hat{\bfy}}_{t_2})}.
\end{align}
Let us evaluate
the equality  \eqref{eq:xtilde2} under the specialization
 $\underline{x}_{1;t_2}=\cdots = \underline{x}_{n;t_2} = 1$.
Then, we see that $F'_{i;t_1}(\underline{\hat{\bfy}}_{t_2})$
is 1 under the specialization.
Therefore, by Lemma \ref{lem:Fcri1},
we have $F'_{i;t_1}(\bfu)=1$.
Then,
again by \eqref{eq:xtilde2}, we obtain
 $G'_{t_1}=
P_{\sigma}$.
 Therefore, by \eqref{eq:dual3}, we have $C'_{t_1}=P_{\sigma}$.
 \par
Finally, let 
$\tilde{\mathbf{\Sigma}}[{t}_0]=\{ \tilde{\Sigma}_t
=(\tilde{\bfx}_t,\tilde{\bfy}_t,B_t)
\mid t\in \bbT_n 
\}$
 be a cluster pattern
with principal coefficients at $t_0$.
We apply
 the separation formula \eqref{eq:sep2}
 for $\tilde{\bfy}_{t_1}$
{\em with the initial point $t_2$}.
Since $C'_{t_1}=P_{\sigma}$ and $F'_{i;t_1}(\bfu)=1$ as shown above,
we obtain $\tilde{\bfy}_{t_1}=\sigma \tilde{\bfy}_{t_2}$.
This means $C_{t_1}=\sigma C_{t_2}$ by \eqref{eq:pr1}.
\end{proof}

In contrast to $x$-variables, the implication (d) $\Longrightarrow$ (a), (b) in Theorem \ref{thm:basic1} does not holds
in general.
In other words, $y$-variables may admit a finer periodicity
than $x$-variables.
 For example, in the extreme case $\mathbb{P}=\mathbf{1}$,
we have $\bfy_t=(1,\dots,1)$ for any $t\in \mathbb{T}_n$. Then, $\bfy_t$ is periodic by any single mutation for any $B$-pattern $\mathbf{B}$.
To provide a sufficient condition that  (d) $\Longrightarrow$ (a), (b) holds,
we introduce some notion.
\begin{defn}
Let 
${{\mathbf{\Sigma}}}=\{ {\Sigma}_t
=({\bfx}_t,
{\bfy},B_t)
\mid t\in \bbT_n 
\}$
be a cluster pattern  with  coefficients in 
any  semifield $\mathbb{P}$,
and
$\tilde{\mathbf{\Sigma}}[t'_0]=\{ \tilde{\Sigma}_t
=(\tilde{\bfx}_t,\tilde{\bfy}_t,B_t)
\mid t\in \bbT_n 
\}$
 be a cluster pattern with principal
coefficients at  $t'_0$,
such that they share a common $B$-pattern $\mathbf{B}$.
Let $\langle \bfy_{t'_0}\rangle$ be the subsemifield 
in  $\mathbb{P}$
generated by $ \bfy_{t'_0}$.
Then, we say that {\em $\mathbf{\Sigma}$ covers 
$\tilde{\mathbf{\Sigma}}[t'_0]$} if there is a semifield homomorphism
such that
\begin{align}
\label{eq:semihom4}
\begin{matrix}
\varphi:&\langle \bfy_{t'_0}\rangle&\rightarrow 
&\mathrm{Trop}(\tilde{\bfy}_{t'_0})\\
&y_{i;t'_0}
&
\mapsto
&
\quad
\tilde{y}_{i;t'_0}
\quad .
\end{matrix}
\end{align}
(Here we use the symbol $t'_0$ so that it cannot be confused with
the initial point $t_0$ in Theorem \ref{thm:basic1}.)
\end{defn}

 Note that for such $\varphi$ in \eqref{eq:semihom4}, 
\begin{align}
\label{eq:semihom5}
\varphi: y_{i;t} \mapsto 
\tilde{y}_{i;t}
\end{align}
holds for any $t \in \mathbb{T}_n$ and $i=1,\dots,n$
by homomorphism property.

\begin{ex}
A cluster pattern with universal coefficients $\mathbf{\Sigma}$
covers a cluster pattern with principal coefficients 
$\tilde{\mathbf{\Sigma}}[t'_0]$ at any point  $t'_0$,
where $\varphi:\mathbb{Q}_{\mathrm{sf}}(\bfy_{t'_0})
\rightarrow \mathrm{Trop}(\tilde{\bfy}_{t'_0})$ is the tropicalization homomorphism.
\end{ex}

Now the second half of the statement of the $xy/GC$ synchronicity  is as follows:

\begin{thm}[$xy/GC$ synchronicity]

\label{thm:basic2}
Let
${{\mathbf{\Sigma}}}=\{ {\Sigma}_t
=({\bfx}_t,
{\bfy}_t,
\allowbreak
B_t)
\mid t\in \bbT_n 
\}$
be a cluster pattern  with  coefficients in 
 any  semifield $\mathbb{P}$
which covers a cluster pattern
$\tilde{\mathbf{\Sigma}}[t'_0]=\{ \tilde{\Sigma}_t
=(\tilde{\bfx}_t,\tilde{\bfy}_t,B_t)
\mid t\in \bbT_n 
\}$
with principal coefficients at some $t'_0$.
Then, Condition (d) implies Conditions (a)--(c) in Theorem
\ref{thm:basic1}.
\end{thm}
\begin{proof}
We show (d) $\Longrightarrow$ (b) in Theorem \ref{thm:basic1}.
Assume that $\bfy_{t_1}=\sigma \bfy_{t_2}$.
By assumption, 
there is a semifield homomorphism $\varphi$ in 
\eqref{eq:semihom4}.
Then, 
by  \eqref{eq:semihom5}, we obtain $\tilde{\bfy}_{t_1}=\sigma \tilde{\bfy}_{t_2}$.
Let us consider the $FGC$-pattern  $\mathbf{\Gamma}(\mathbf{B},t'_0)=
\{ \Gamma'_t=(\mathbf{F'}_{t}(\bfu), G'_t, C'_t, B_t)
\mid t\in \mathbb{T}_n \}$
of $\mathbf{B}$ with the initial point $t'_0$.
Then, by \eqref{eq:pr1}, we have $C'_{t_1}=\sigma C'_{t_2}$.
Let 
$\tilde{\mathbf{\Sigma}}[{t}_0]=\{ \tilde{\Sigma}'_t
=(\tilde{\bfx}'_t,\tilde{\bfy}'_t,B_t)
\mid t\in \bbT_n 
\}$
 be a cluster pattern with principal
coefficients at  ${t}_0$.
Then, applying Theorem \ref{thm:basic1}
 for $\tilde{\mathbf{\Sigma}}[{t}_0]$ and $\mathbf{\Gamma}(\mathbf{B},t'_0)$,
 we obtain $\tilde{\bfy}'_{t_1}=\sigma \tilde{\bfy}'_{t_2}$.
Then, applying  \eqref{eq:pr1} for $\tilde{\bfy}'_{t}$, we obtain
$C_{t_1}=\sigma C_{t_2}$.
\end{proof}

We obtain the following property of the periodicities of $x$- and $y$-variables.

\begin{cor}
\label{cor:comd1}
Let $D$ be a common skew-symmetrizer of $\mathbf{B}$.
\par
(1). Let 
${{\mathbf{\Sigma}}}=\{ {\Sigma}_t
=({\bfx}_t,
{\bfy}_t,B_t)
\mid t\in \bbT_n 
\}$
be a cluster pattern  with  coefficients in 
any  semifield $\mathbb{P}$.
If $\bfx_{t_1}=\sigma \bfx_{t_2}$ occurs for some $t_1,t_2\in \mathbb{T}_n$
 and a permutation $\sigma\in S_n$,
 then $\sigma$ is compatible with $D$.
 \par
 (2).
 Let
${{\mathbf{\Sigma}}}=\{ {\Sigma}_t
=({\bfx}_t,
{\bfy}_t,
\allowbreak
B_t)
\mid t\in \bbT_n 
\}$
be a cluster pattern  with  coefficients in 
 any  semifield $\mathbb{P}$
which covers a cluster pattern
with principal coefficients at some $t'_0$.
If $\bfy_{t_1}=\sigma \bfy_{t_2}$ occurs for some $t_1,t_2\in \mathbb{T}_n$
 and a permutation $\sigma\in S_n$,
 then $\sigma$ is compatible with $D$.
\end{cor}
\begin{proof}
This follows from Proposition \ref{prop:compat1}
and Theorems \ref{thm:basic1} and \ref{thm:basic2}.
\end{proof}

\subsection{More synchronicity results}
\label{subsec:consequences}
Let us present some consequences of Theorems \ref{thm:basic1} and
\ref{thm:basic2}.
Below ``periodicity" is used in the sense of $\sigma$-periodicity.

The following theorem states that  the periodicity of
$x$-variables is independent of the choice of $\mathbb{P}$,
$\bfx_t$'s, and $\bfy_t$'s.
It was proved by \cite{Cao18} using $d$-vectors.
Here we give an alternative proof via Theorem \ref{thm:basic1}.
\begin{thm}
[{\cite[Proposition 6.1]{Cao18}}]
\label{thm:x1}
Let
${{\mathbf{\Sigma}}}=\{ {\Sigma}_t
=({\bfx}_t,
{\bfy}_t,B_t)
\mid t\in \bbT_n 
\}$
be a cluster pattern with  coefficients in 
any  semifield $\mathbb{P}$.
Then, the periodicity of $x$-variables $\bfx_t$  depends only on 
the $B$-pattern $\mathbf{B}$ therein.
\end{thm}
\begin{proof}
By Theorem \ref{thm:basic1}, the periodicity of $x$-variables for any 
cluster pattern coincides with the periodicity of $C$-matrices
(with any initial point),
which are uniquely determined from $\mathbf{B}$.
\end{proof}

Similarly, by  Theorem 
\ref{thm:basic2}, we obtain the counterpart for $y$-variables.

\begin{thm}
\label{thm:y1}
Let
${{\mathbf{\Sigma}}}=\{ {\Sigma}_t
=({\bfx}_t,
{\bfy}_t,B_t)
\mid t\in \bbT_n 
\}$
be a cluster pattern with  coefficients in 
any  semifield $\mathbb{P}$
such that
 $\mathbf{\Sigma}$ covers 
a cluster pattern with principal coefficients at 
some $t'_0$.
Then, the periodicity of $y$-variables $\bfy_t$  depends only on 
the $B$-pattern $\mathbf{B}$ therein.
\end{thm}

Combining Theorems \ref{thm:x1} and \ref{thm:y1} (and their proofs), we have the following
synchronicity of $x$- and $y$-variables.

\begin{thm}
\label{thm:xy1}
The  $x$-variables in Theorem \ref{thm:x1}
and the  $y$-variables in Theorem \ref{thm:y1}
with a common $B$-pattern $\mathbf{B}$
share the same periodicity.
\end{thm}

\begin{ex}
As a special case of Theorem \ref{thm:xy1},
take a cluster pattern $\underline{\mathbf{\Sigma}}=\{ \underline{\Sigma}_t
=(\underline{\bfx}_t,B_t)
\mid t\in \bbT_n 
\}$ without coefficients for Theorem \ref{thm:x1}
and a cluster pattern $\overline{\mathbf{\Sigma}}=\{ \overline{\Sigma}_t
=(\overline{\bfx}_t,\overline{\bfy}_t,B_t)
\mid t\in \bbT_n 
\}$ with universal coefficients for Theorem \ref{thm:y1}.
Then, we have
\begin{align}
\underline{\bfx}_{t_1}=\sigma \underline{\bfx}_{t_2}
\quad
\Longleftrightarrow
\quad
\overline{\bfy}_{t_1}=\sigma \overline{\bfy}_{t_2}.
\end{align}
This proves the synchronicity \eqref{eq:xy1} in Question \ref{question:xy1}
as the special case $\sigma=\mathrm{id}$.
\end{ex}

So far, we only consider the synchronicity  of $x$-variables
and $y$-variables. Let us extend the result for seeds and $Y$-seeds.

The following theorem tells that the periodicities of $x$-variables alone and seeds
including them coincide for any cluster pattern.
It was proved by \cite{Cao18} using $d$-vectors.
Here we give an alternative proof via Theorem \ref{thm:basic1}.
See also Conjecture
\ref{conj:f1} and Theorem \ref{thm:f1} for related results.

\begin{thm}[{\cite[Proposition 6.1]{Cao18}}]
\label{thm:seed1}
Let
${{\mathbf{\Sigma}}}=\{ {\Sigma}_t
=({\bfx}_t,
{\bfy}_t,B_t)
\mid t\in \bbT_n 
\}$
be a cluster pattern with  coefficients in 
any  semifield $\mathbb{P}$.
Then,   the $x$-variables $\bfx_t$
and the seeds $\Sigma_t$
share the same periodicity.
Furthermore, it depends only on 
the $B$-pattern $\mathbf{B}$ therein.
\end{thm}

\begin{proof}
It is enough to show the following claim:
\begin{align}
{\bfx}_{t_1}=\sigma {\bfx}_{t_2}
\quad
\Longrightarrow
\quad
{\bfy}_{t_1}=\sigma {\bfy}_{t_2},
B_{t_1}=\sigma B_{t_2}.
\end{align}
This is a consequence of  Theorem \ref{thm:basic1}
and \eqref{eq:dual2}.
\end{proof}

A pair $(\bfy_t,B_t)$ is called a $Y$-seed in \cite{Fomin07}.
We have a counterpart of Theorem \ref{thm:seed1}
for $Y$-seeds.

\begin{thm}
\label{thm:yseed1}
Let
${{\mathbf{\Sigma}}}=\{ {\Sigma}_t
=({\bfx}_t,
{\bfy}_t,B_t)
\mid t\in \bbT_n 
\}$
be a cluster pattern with  coefficients in 
any  semifield $\mathbb{P}$
such that $\mathbf{\Sigma}$ covers 
a cluster pattern   with principal coefficients at 
some $t'_0$.
Then, the $y$-variables $\bfy_t$ 
and the $Y$-seeds $(\bfy_t,B_t)$ 
share the same periodicity.
Furthermore, it depends only on 
the $B$-pattern $\mathbf{B}$ therein.
\end{thm}
\begin{proof}
It is enough to show the following claim:
\begin{align}
{\bfy}_{t_1}=\sigma {\bfy}_{t_2}
\quad
\Longrightarrow
\quad
B_{t_1}=\sigma B_{t_2}.
\end{align}
This is a consequence of  Theorem \ref{thm:basic2}
and \eqref{eq:dual2}.
\end{proof}

By combining  Theorems \ref{thm:xy1}, \ref{thm:seed1}, and \ref{thm:yseed1},
we also have the following synchronicity of
seeds and $Y$-seeds.

\begin{thm}
\label{thm:xyseed1}
Let
${{\mathbf{\Sigma}}}=\{ {\Sigma}_t
=({\bfx}_t,
{\bfy}_t,B_t)
\mid t\in \bbT_n 
\}$
be a cluster pattern with  coefficients in 
 any semifield $\mathbb{P}$
such that $\mathbf{\Sigma}$ covers 
a cluster pattern with principal coefficients at 
some $t'_0$.
Then, the seeds $\Sigma_t$ and the $Y$-seeds
$(\bfy_t,B_t)$ share the same periodicity.
Furthermore, it depends only on
the $B$-pattern $\mathbf{B}$ therein.
\end{thm}

\subsection{Conjectures by Fomin and Zelevinsky}

In \cite{Fomin03a, Fomin07} 
Fomin and Zelevinsky   proposed several important  conjectures on 
 periodicity  properties of cluster patterns.

Let us recall some definitions from \cite{Fomin03a, Fomin07}.
\begin{defn}
Let
${{\mathbf{\Sigma}}}=\{ {\Sigma}_t
=({\bfx}_t,
{\bfy}_t,B_t)
\mid t\in \bbT_n 
\}$
be a cluster pattern with  coefficients in 
any  semifield $\mathbb{P}$.
An {\em unlabeled cluster} of $\mathbf{\Sigma}$ 
 is a set
 $\{x_{1;t}, \dots, x_{n;t}\}$ 
 ($t\in \mathbb{T}_n$)
 of $x$-variables belonging to $\bfx_t$.
An {\em unlabeled seed} $[\Sigma_t]$ of $\mathbf{\Sigma}$
 ($t\in \mathbb{T}_n$)
 is an equivalence class of 
all labeled seeds 
obtained from $\Sigma_t$
by the action of permutations given by \eqref{eq:s4}--\eqref{eq:s6}.
The {\em exchange graph (of unlabeled seeds) of  $\mathbf{\Sigma}$}
is a quotient graph of $\mathbb{T}_n$ modulo
the equivalence relation $t\sim t'$ defined by the condition 
$[\Sigma_t]=[\Sigma_{t'}]$.
\end{defn}

The following conjecture was proposed by \cite{Fomin03a}.

\begin{conj}[{\cite[Section 1.5]{Fomin03a}}]
\label{conj:f1}
Let
${{\mathbf{\Sigma}}}=\{ {\Sigma}_t
=({\bfx}_t,
{\bfy}_t,B_t)
\mid t\in \bbT_n 
\}$
be a cluster pattern with  coefficients in 
any  semifield $\mathbb{P}$.
Then, the following properties hold:
\par
(a).
Each unlabeled cluster
 $\{x_{1;t}, \dots, x_{n;t}\}$  uniquely determines an unlabeled seed containing it.
 \par
(b). The exchange graph of $\mathbf{\Sigma}$
  depends only on 
the $B$-pattern $\mathbf{B}$ therein.
\par
(c). 
Two unlabeled clusters are adjacent in the exchange graph
if and only if they have exactly $n-1$ common $x$-variables.
\end{conj}

Claim (a) was proved in several cases in \cite{Fomin03a,Buan05c,Fomin08, Gekhtman07,Cerulli12}
and in full generality in \cite{Cao18}.
Claim (b) was also proved 
in several cases
 in \cite{Fomin03a,Fomin08, Gekhtman07, Cerulli12}
 and in full generality in \cite{Cao18}.
It was shown in \cite{Gekhtman07} 
that Claim (c) follows from Claim (a).

\begin{thm}
[{\cite[Proposition 6.1]{Cao18}}]
\label{thm:f1}
Conjecture \ref{conj:f1} is true.
\end{thm}
\begin{proof}
Claim (a) is equivalent to the first half of Theorem
\ref{thm:seed1}.
Claim (b) follows from the second half of Theorem
\ref{thm:seed1}.
Claim (c) follows from Claim (a) due to  \cite[Theorem 5]{Gekhtman07}
as already mentioned.
\end{proof}

\begin{rem}
Conjecture \ref{conj:f1} (b) was
enhanced by \cite[Conjeture 4.3]{Fomin07}  to
 includes  the statement that
 the {\em canonical connection} on the exchange graph (of unlabeled seeds)
 of $\mathbf{\Sigma}$
 depends only on $\mathbf{B}$.
This is also true by the following reason.
 The exchange graph of {\em labeled seeds} has
the canonical connection induced from the labeling of $\mathbb{T}_n$.
The canonical connection on the exchange graph of {\em unlabeled seed},
 is further induced from
the above one modulo $\sigma$-periodicity. (This is well defined by \eqref{eq:comp1}.)
Since the  $\sigma$-periodicity  depends only on $\mathbf{B}$,
the induced canonical connection also depends only on $\mathbf{B}$.
\end{rem}

There is another conjecture  proposed by \cite{Fomin07}.

\begin{conj}[{\cite[Conjetures 4.7 \& 4.8]{Fomin07}}]
\label{conj:f2}
Let 
$\tilde{\mathbf{\Sigma}}[{t}_0]=\{ \tilde{\Sigma}_t
=(\tilde{\bfx}_t,\tilde{\bfy}_t,
\allowbreak
B_t)
\mid t\in \bbT_n 
\}$
be  a cluster pattern with
principal coefficients 
at any ${t}_0$.
Then, the following properties hold:
\par
(a).
The seeds $\tilde{\Sigma}_t$ and the $Y$-seeds
$(\tilde{\bfy}_{t},B_{t})$ share the same periodicity.
\par
(b).
There are  finitely many distinct seeds $\Sigma_t$ if and only if
there are  finitely many  distinct $Y$-seeds $(\tilde{\bfy}_t,B_t)$.
\end{conj}

\begin{thm}
Conjecture \ref{conj:f2} is true.
\end{thm}
\begin{proof}
Claim (a) is a special case of Theorem \ref{thm:xyseed1}.
Claim (b) is a consequence of Claim (a) as already noted in \cite{Fomin07}.
\end{proof}

\subsection{Synchronicity under dualities}
\label{subsec:dualities}
As another application of 
Theorems \ref{thm:basic1} and \ref{thm:seed1},
let us show that the dualities of $B$-patterns
considered by Fock and Goncharov in
 \cite[Section 1]{Fock03}
preserve the periodicity of
 cluster patterns.

Let $\mathbf{B}=\{ B_t \mid t\in \mathbb{T}_n\}$ be a $B$-pattern with a common skew-symmetrizer $D$. Then, it is easy to see from \eqref{eq:bmut1} that the following three families of matrices are also $B$-patterns.

{\em (a). (Transposition duality)} $\mathbf{B}^T:=\{ B^T_t \mid t\in \mathbb{T}_n\}$.
A common skew-symmetrizer is given by $D'=\tilde{d} D^{-1}$, where $\tilde{d}$
 is the least common multiple of $d_1,\dots,d_n$.

{\em (b). (Chiral duality)} $-\mathbf{B}:=\{ -B_t \mid t\in \mathbb{T}_n\}$.
A common skew-symmetrizer is given by $D$.
 
 {\em (c). (Langlands duality)} $-\mathbf{B}^T:=\{ -B^T_t \mid t\in \mathbb{T}_n\}$.
A common skew-symmetrizer is given by $D'$ in (a).

 Note that the compatibility of 
  a permutation $\sigma\in S_n$ with $D$
  is equivalent to the compatibility with $D'$.

Let us start with the transposition duality.
The following theorem was suggested by Sergey Fomin to the author.
\begin{thm}
\label{thm:dual4}
Let
${{\mathbf{\Sigma}}}=\{ {\Sigma}_t
=({\bfx}_t,
{\bfy}_t,B_t)
\mid t\in \bbT_n 
\}$
and
${{\mathbf{\Sigma}}}'=\{ {\Sigma}'_t
=({\bfx}'_t,
{\bfy}'_t,B^T_t)
\mid t\in \bbT_n 
\}$
be  cluster patterns with coefficients
in any semifields $\mathbb{P}$ and $\mathbb{P}'$, respectively.
Then, the seeds in ${\mathbf{\Sigma}}$ and the seeds in
${\mathbf{\Sigma}}'$ share the same periodicity.
\end{thm}
\begin{proof}
Thanks to Theorem \ref{thm:seed1}, 
it is enough to prove
\begin{align}
\label{eq:xperiod2}
{\bfx}_{t_1}=\sigma {\bfx}_{t_2}
\quad
\Longleftrightarrow
\quad
{\bfx}'_{t_1}=\sigma {\bfx}'_{t_2},
\end{align}
where $\sigma$ is compatible with $D$ and $D'$.
Let us consider the pair of $FGC$-patterns $\mathbf{\Gamma}(\mathbf{B},t_2)
=\{ \Gamma_t=(\mathbf{F}_{t}(\bfu), G_t, C_t, B_t)
\mid t\in \mathbb{T}_n \}$
and 
 $\mathbf{\Gamma}(\mathbf{B}^T,{t}_1)=
\{ \Gamma'_t=(\mathbf{F'}_{t}(\bfu), G'_t, C'_t, B^T_t)
\mid t\in \mathbb{T}_n \}$,
where $t_1$ and $t_2$ are the ones in
\eqref{eq:xperiod2}.
We recall  another duality of $G$- and $C$-matrices from \cite[Eq.~(1.13)]{Nakanishi11a}, namely,
\begin{align}
\label{eq:dual4}
C_{t_1}^T=G'_{t_2}.
\end{align}
(This is also a consequence of Theorem \ref{thm:sign}.)
Let us assume that ${\bfx}_{t_1}=\sigma {\bfx}_{t_2}$.
Then,
by Theorem \ref{thm:basic1}, we have
$C_{t_1}=P_{\sigma}$.
Therefore, by \eqref{eq:dual4}, we have $G'_{t_2}=P_{\sigma^{-1}}$.
Then, 
again by Theorem \ref{thm:basic1}, we
have ${\bfx}'_{t_2}=\sigma^{-1} {\bfx}'_{t_1}$,
that is, ${\bfx}'_{t_1}=\sigma {\bfx}'_{t_2}$.
The other direction also follows by symmetry.
\end{proof}

Next we consider the chiral duality.

\begin{thm}
\label{thm:dual5}
Let
${{\mathbf{\Sigma}}}=\{ {\Sigma}_t
=({\bfx}_t,
{\bfy}_t,B_t)
\mid t\in \bbT_n 
\}$
and
${{\mathbf{\Sigma}}}'=\{ {\Sigma}'_t
=({\bfx}'_t,
{\bfy}'_t,-B_t)
\mid t\in \bbT_n 
\}$
be  cluster patterns with coefficients
in any semifields $\mathbb{P}$ and $\mathbb{P}'$, respectively.
Then, the seeds in ${\mathbf{\Sigma}}$ and the seeds in
${\mathbf{\Sigma}}'$ share the same periodicity.
\end{thm}
\begin{proof}
One can  prove the claim  by 
Theorem  \ref{thm:xyseed1} and
the following involution between $Y$-patterns  for $\mathbf{B}$
and $-\mathbf{B}$ with universal $y$-variables
\cite[Proof of Proposition 5.3]{Fomin07}:
\begin{align}
(\bfy_{t},B_{t}) \leftrightarrow (\bfy_{t}^{-1},-B_{t}) .
\end{align}
Alternatively, one can also prove it using
the following equality \cite[Theorem 2.8]{Fujiwara18} for
 the pair of $FGC$-patterns $\mathbf{\Gamma}(\mathbf{B},t_0)
=\{ \Gamma_t=(\mathbf{F}_{t}(\bfu), G_t, C_t, B_t)
\mid t\in \mathbb{T}_n \}$
and 
 $\mathbf{\Gamma}(-\mathbf{B},{t}_0)=
\{ \Gamma'_t=(\mathbf{F'}_{t}(\bfu), G'_t, C'_t, -B_t)
\mid t\in \mathbb{T}_n \}$ with any initial point $t_0$:
\begin{align}
\label{eq:CC1}
C'_t = C_t +F_t B_t,
\quad
F'_{i;t}(\bfu)
=F_{i;t}(\bfu),
\end{align}
where $F_t$ (the {\em $F$-matrix at $t$}) is the degree matrix of the $F$-polynomials $\mathbf{F}_{t}(\bfu)$ defined in
\cite{Fujiwara18}. 
(This is {\em not}  a consequence of Theorem \ref{thm:sign}.)
By \eqref{eq:CC1}, \eqref{eq:dual2} and  \eqref{eq:gf2}, we obtain
\begin{align}
C_{t_1} = \sigma C_{t_2}
\quad
\Longleftrightarrow
\quad
C'_{t_1} = \sigma C'_{t_2}.
\end{align}
Then, the claim follows from 
Theorems \ref{thm:basic1} and \ref{thm:seed1}.
\end{proof}

Finally we consider the Langlands duality.

\begin{thm}
\label{thm:dual6}
Let
${{\mathbf{\Sigma}}}=\{ {\Sigma}_t
=({\bfx}_t,
{\bfy}_t,B_t)
\mid t\in \bbT_n 
\}$
and
${{\mathbf{\Sigma}}}'=\{ {\Sigma}'_t
=({\bfx}'_t,
{\bfy}'_t,-B^T_t)
\mid t\in \bbT_n 
\}$
be  cluster patterns with coefficients
in  any semifields $\mathbb{P}$ and $\mathbb{P}'$, respectively.
Then, the seeds in ${\mathbf{\Sigma}}$ and the seeds in
${\mathbf{\Sigma}}'$ share the same periodicity.
\end{thm}
\begin{proof}
This is proved by combining 
Theorems \ref{thm:dual4} and \ref{thm:dual5}.
Alternatively, one can also prove it using
the following yet another duality \cite[Eq.~(1.11)]{Nakanishi11a} for
 the pair of $FGC$-patterns $\mathbf{\Gamma}(\mathbf{B},t_0)
=\{ \Gamma_t=(\mathbf{F}_{t}(\bfu), G_t, C_t, B_t)
\mid t\in \mathbb{T}_n \}$
and 
 $\mathbf{\Gamma}(-\mathbf{B}^T,{t}_0)=
\{ \Gamma'_t=(\mathbf{F'}_{t}(\bfu), G'_t, C'_t, -B^T_t)
\mid t\in \mathbb{T}_n \}$ with any initial point $t_0$:
\begin{align}
\label{eq:GC1}
G_t^T C'_t=I.
\end{align}
(This is also a consequence of Theorem \ref{thm:sign}.
Also, this  is actually equivalent to \eqref{eq:dual1}
due to the relation $C'_t = D C_t D^{-1}$.)
By \eqref{eq:GC1},
we obtain
\begin{align}
G_{t_1} = \sigma G_{t_2}
\quad
\Longleftrightarrow
\quad
C'_{t_1} = \sigma C'_{t_2}.
\end{align}
Then, the claim follows from 
Theorems \ref{thm:basic1} and \ref{thm:seed1}.
\end{proof}

\subsection{Synchronicity for conjugate pair}
\label{subsec:conjugate1}
Let $R=\mathrm{diag}(r_1,\dots,r_n)$ be
a diagonal matrix with positive integer diagonal entries
$r_1,\dots,r_n$.
Let us consider a pair of $B$-patterns of the form
\begin{align}
R\mathbf{B}&=\{ RB_t \mid t\in \mathbb{T}_n\},
\\
\mathbf{B}R&=\{ B_t R \mid t\in \mathbb{T}_n\},
\end{align}
where $RB_t$ and $B_tR$ are (skew-symmetrizable) integer matrices, but $B_t$ is not necessarily an integer matrix.
They are related by the conjugation
\begin{align}
\label{eq:conj1}
B_t R = R^{-1} (RB_t) R.
\end{align}
It is easy to see from \eqref{eq:conj1} and  \eqref{eq:bmut1} that
if one is a $B$-pattern, then so is the other.
We call such $B$-patterns a {\em conjugate pair}.
Such a pair appears in the context of generalized cluster algebras
\cite{Nakanishi14a,Nakanishi15}.
See Definition \ref{defn:companion1}.

\begin{lem} 
\label{lem:ss1}
Let $R\mathbf{B}$ and $\mathbf{B}R$ be a conjugate pair.
If $D$ is a common skew-symmetrizer of $R\mathbf{B}$,
then $DR^2$ is a common skew-symmetrizer of $\mathbf{B}R$.
\end{lem}
\begin{proof}
Suppose that $(DRB_t)^{T}=-DRB_t$. Then, $R(DRB_t)^{T}R=-DR^2B_t R$.
That is, $(DR^2B_t R)^{T}=-DR^2B_t R$.
\end{proof}

\begin{prop}
\label{prop:conj2}
 Let $R\mathbf{B}$ and $\mathbf{B}R$ be a conjugate pair.
Let 
$\mathbf{\Gamma}(R\mathbf{B},t_0)=
\{ \Gamma_t=(\mathbf{F}_{t}(\bfu), G_t, C_t, RB_t)
\mid t\in \mathbb{T}_n \}$
and 
$\mathbf{\Gamma}(\mathbf{B}R,t_0)=
\{ \Gamma'_t=(\mathbf{F}'_{t}(\bfu), G'_t, C'_t, \allowbreak B_tR)
\mid t\in \mathbb{T}_n \}$
be their $FGC$-patterns with any initial point $t_0$.
\par
(1).
The following equalities hold:
\begin{align}
\label{eq:conjg1}
G'_t &= R^{-1} G_t R,\\
\label{eq:conjc1}
C'_t &= R^{-1} C_t R.
\end{align}
\par
(2).
If $ G_{t_1}=\sigma G_{t_2}$ or $ G'_{t_1}=\sigma G'_{t_2}$
 for some $t_1,t_2\in \mathbb{T}_n$
and  a permutation $\sigma\in S_n$,
then $\sigma$ is compatible with $D$ and $R$.
\end{prop}
\begin{proof}
(1).
They are true at $t_0$.
Then, they are shown from \eqref{eq:gmut1}, \eqref{eq:cmut1},
and \eqref{eq:conj1}
by induction on $t$.
\par
(2).
Suppose that $G_{t_1}=\sigma G_{t_2}$.
Then, $\sigma$ is compatible with $D$ by Proposition \ref{prop:compat1}.
Meanwhile,
by \eqref{eq:conjg1},
$R^{-1}G_{t_1}R=G'_{t_1}$ and $R^{-1}G_{t_2}^{-1}R=(G'_{t_2})^{-1}$.
Therefore,
$R^{-1}P_{\sigma}R=R^{-1}G_{t_2}^{-1}G_{t_1}R=(G'_{t_2})^{-1}G'_{t_1}$,
where the righthand side is an integer matrix.
Then, by repeating the  argument in the proof of Proposition \ref{prop:compat1},
we conclude that $\sigma$ is compatible with $R$.
The other case $ G'_{t_1}=\sigma G'_{t_2}$ is similar.
\end{proof}

\begin{thm}
\label{thm:conj1}
Let $R\mathbf{B}$ and $\mathbf{B}R$ be a conjugate pair
with common skew-symmetrizers $D$ and $DR^2$, respectively.
Let
${{\mathbf{\Sigma}}}=\{ {\Sigma}_t
=({\bfx}_t,
{\bfy}_t,RB_t)
\mid t\in \bbT_n 
\}$
and
${{\mathbf{\Sigma}}}'=\{ {\Sigma}'_t
=({\bfx}'_t,
{\bfy}'_t,B_tR)
\mid t\in \bbT_n 
\}$
be  cluster patterns with coefficients
in  any semifields $\mathbb{P}$ and $\mathbb{P}'$, respectively.
\par
(1). 
If $ \Sigma_{t_1}=\sigma \Sigma_{t_2}$ or $ \Sigma'_{t_1}=\sigma \Sigma'_{t_2}$
 for some $t_1,t_2\in \mathbb{T}_n$
and  a permutation $\sigma\in S_n$,
then $\sigma$ is compatible with $D$ and $R$.
\par
(2).
The seeds in ${\mathbf{\Sigma}}$ and the seeds in
${\mathbf{\Sigma}}'$ share the same periodicity.
\end{thm}
\begin{proof}
(1). 
This follows from Theorems 
\ref{thm:basic1} and
\ref{thm:seed1}, and  Proposition \ref{prop:conj2}.
\par
(2). 
Let 
$\mathbf{\Gamma}(R\mathbf{B},t_0)=
\{ \Gamma_t=(\mathbf{F}_{t}(\bfu), G_t, C_t, RB_t)
\mid t\in \mathbb{T}_n \}$
and 
$\mathbf{\Gamma}(\mathbf{B}R,t_0)=
\{ \Gamma'_t=(\mathbf{F}'_{t}(\bfu), G'_t, C'_t, B_tR)
\mid t\in \mathbb{T}_n \}$
be the $FGC$-patterns of $R\mathbf{B}$ and $\mathbf{B}R$
with any initial point $t_0$, respectively.
For a permutation $\sigma$ which is compatible with $R$,
$G_{t_1}=\sigma G_{t_2}$ if and only if $G'_{t_1}=\sigma G'_{t_2}$
thanks to \eqref{eq:conjg1}.
Then, the claim follows from 
Theorems \ref{thm:basic1}, \ref{thm:seed1},
and Claim (1).
\end{proof}

See the forthcoming Corollary
\ref{cor:gconj1} for a stronger version of Theorem \ref{thm:conj1} (1)
when $B$ is an integer matrix.

\section{Open problem}
\label{sec:open}

We have seen that various objects such as seeds, $Y$-seeds, $x$-variables, $y$-variables, $C$-matrices, $G$-matrices, etc,
with a common $B$-pattern $\mathbf{B}$
 share the same periodicity. In other words, 
 they share
 the same exchange graph.
 Also, a categorification of a cluster pattern share the same  exchange graph 
with the underlying cluster pattern by definition.
 In addition, in Part \ref{part:GCA} we show that cluster patterns of
generalized cluster algebras
  also share the  exchange graphs  with ordinary cluster algebras,
  up to the Laurent positivity conjecture (Theorem \ref{thm:gcompanion1}).
 These results suggest that the underlying exchange graph is the essence
 of what we vaguely call {\em cluster structure}.

In this aspect it may be more convenient to consider the  exchange graphs of labeled seeds,
rather than unlabeled seeds.
Let us introduce
the {\em labeled exchange graph} of  a cluster pattern $\mathbf{\Sigma}$,
which is the quotient graph of $\mathbb{T}_n$ modulo
the equivalence relation $t\sim t'$ defined by the condition 
$\Sigma_t=\Sigma_{t'}$,
endowed with arrows  $t_2 \xrightarrow{\sigma}
 t_1$ if $\Sigma_{t_1} = \sigma \Sigma_{t_2}$ for a permutation $\sigma\in S_n$.
Moreover, these labeled exchange graphs are better regarded as {\em groupoids} as proposed earlier in \cite[Section 2.1]{Fock07}, where objects are seeds, $Y$-seeds, etc., and 
morphisms are generated by their mutations and  permutations.
Then, the synchronicity  means that we have {\em different realizations (representations)  of a common groupoid}  with various objects.
It is natural to call this common underlying groupoid $\mathcal{G}(\mathbf{B})$ the {\em cluster groupoid} of $\mathbf{B}$. (It is called 
the {\em cluster modular groupoid}
in \cite{Fock03}.)

In analogy with reflection groups and Coxeter groups, the usual construction of 
 $\mathcal{G}(\mathbf{B})$ by
seeds and their mutations  
may correspond to an  explicit construction of  reflection groups
starting from the simple roots.
 On the other hand, like Coxeter groups,  cluster groupoids should be also  defined as abstract groupoids by generators and relations.
Unfortunately we do not yet know  much about
the fundamental relations of morphisms in  cluster groupoids except for  rank 2 case \cite{Fomin02,Fock03}
and  certain surface type \cite[Theorem 9.1, Remark 9.19]{Fomin08}.
Therefore, we believe that the following problem has fundamental importance
in cluster algebra theory:

\begin{problem}
Describe the fundamental relations of morphisms in a cluster groupoid
$\mathcal{G}(\mathbf{B})$
in terms of the exchange matrices in $\mathbf{B}$.
\end{problem}

\newpage

\part{Synchronicity in cluster patterns of  generalized cluster algebras}
\label{part:GCA}


{\em Generalized cluster algebras} (GCAs) were introduced by Chekhov and Shapiro
\cite{Chekhov11} as a generalization of (ordinary) cluster algebras (CAs).
 It turned out that many (possibly all) important properties of
CAs are naturally extended to GCAs \cite{Chekhov11, Nakanishi14a}.
In this part we show that 
all main results in Part \ref{part:CA} are naturally 
extended to cluster patterns of GCA as expected,
up to the Laurent positivity which is still a conjecture for GCA. Moreover,
we show that any  cluster pattern of GCA synchronizes with its ``companion cluster patterns" of CA.

\section{Cluster patterns and separation formulas of GCA}
\label{sec:cluster2}

\subsection{Cluster patterns of GCA}
We recall the definitions of seeds and cluster patterns of (normalized) GCA
\cite{Chekhov11}
following the formulation of \cite{Nakanishi14a, Nakanishi15}.

\begin{defn}[Seeds of GCA]
\label{defn:gseed1}
Let $\bbP$ be any semifield.
We fix a
diagonal matrix  $R=\mathrm{diag}(r_1,\dots,r_n)$ whose diagonal entries
$r_1,\dots,r_n$ are positive integers.
(Here we use the notation $r_i$ instead of $d_i$ 
and {\em vice versa} in \cite{Nakanishi14a, Nakanishi15}.)
We call it the {\em degree} of GCA.
A {\em (labeled) seed of degree $R$  with coefficients in $\bbP$\/}
is a quadruplet
$\Sigma=(\bfx,\bfy,\bfz,B)$,
where 
$(\bfx,\bfy,B)$
is a seed of CA in Definition \ref{defn:seed1},
and
\begin{align}
\bfz&=(z_{i,s})_{1\leq i \leq n, 1 \leq s \leq r_i-1}\\
&= (\bfz_1,\dots,\bfz_n),
\quad \bfz_i=(z_{i,s})_{1\leq s\leq r_i-1}
\end{align}
are collection of arbitrary elements in $\mathbb{P}$.
For $\bfz$, we associate
 an $n$-tuple $\bfZ(\eta)=(Z_1(\eta),\dots,Z_n(\eta))$
 of 
polynomials ({\em exchange polynomials}) in a single variable $\eta$
with coefficients in $\mathbb{P}$ such that
\begin{align}
Z_i(\eta)=\sum_{s=0}^{r_i} z_{i,s}\eta^s,
\quad (r_i=\deg Z_i(\eta))
\end{align}
with  $z_{i,0}=z_{i,r_i}=1$.
\end{defn}

When $r_i=1$, $\bfz_i$ is empty, and $Z_i(\eta)=1+\eta$.
In particular,
when $R=I$, $\bfz$ is empty, 
and seeds of GCA reduce to seeds of CA.
For $R\neq I$, the seed mutations in Definition 
\ref{defn:mutation1} are generalized as follows.

\begin{defn}[Seed mutations of GCA]
\label{defn:gmutation1}
For any seed $\Sigma=(\bfx,\bfy,\bfz,B)$ of degree $R$ 
and any $k\in \{1,\dots,n\}$,
we define a new seed
$\Sigma'=(\bfx',\bfy',\bfz',B')$
 of the same degree
by the following formulas:
\begingroup
\allowdisplaybreaks
\begin{align}
\label{eq:gxmut1}
x'_i
&=
\begin{cases}
\displaystyle
x_k^{-1}\left( \prod_{j=1}^n x_j^{[-b_{jk}]_+}
\right)^{r_k}
\frac{ Z_k(\hat{y}_k)}{ Z_k\vert_{\mathbb{P}}({y}_k)}
& i=k,
\\
x_i
&i\neq k,
\end{cases}
\\
\label{eq:gymut1}
y'_i
&=
\begin{cases}
\displaystyle
y_k^{-1}
& i=k,
\\
y_i 
\left(
y_k^{[b_{ki}]_+} 
\right)^{r_k}Z_k\vert_{\mathbb{P}}({y}_k)^{-b_{ki}}
&i\neq k,
\end{cases}
\\
\label{eq:gzmut1}
z'_{i,s}
&=
\begin{cases}
 z_{k,r_k-s}
& i = k,
\\
z_{i,s}
& i \neq k,
\end{cases}
\\
\label{eq:gbmut1}
b'_{ij}&=
\begin{cases}
-b_{ij}
&
\text{$i=k$ or $j=k$,}
\\
b_{ij}+
r_k\left(
b_{ik} [b_{kj}]_+
+
[-b_{ik}]_+b_{kj}
\right)
&
i,j\neq k,
\end{cases}
\end{align}
where $\hat{y}_k$ in \eqref{eq:gxmut1}
is defined by
 \begin{align}
\label{eq:gyhat1}
\hat{y}_i
:=y_i \prod_{j=1}^n x_j^{b_{ji}}
\in \mathcal{F}.
\end{align}
\endgroup
 The seed $\Sigma'$ is called the {\em mutation of $\Sigma$ at direction $k$},
and denoted by $\mu_k(\Sigma)$. 
Again, these mutations are involutive.
\end{defn}

It is easy to see that $RB'$ (resp. $B'R$) is obtained from $RB$ (resp. $BR$) by the ordinary matrix mutation \eqref{eq:bmut1}.
In other words, the (ordinary) $B$-patterns $R\mathbf{B}$ and 
$\mathbf{B}R$ form a conjugate pair in the sense of Section
\ref{subsec:conjugate1}.

In terms of exchange polynomials,
the mutation of $\bfz$ in \eqref{eq:gzmut1} is rephrased as
\begin{align}
\label{eq:gZmut1}
Z'_i(\eta)
&=
\begin{cases}
\displaystyle
\overline{Z}_k(\eta):=
\sum_{s=0}^{r_k} z_{k,r_k-s}\eta^s
& i = k,
\\
Z_i(\eta)
& i \neq k.
\end{cases}
\end{align}

Accordingly, a {\em cluster pattern $\mathbf{\Sigma}_R=\{ \Sigma_t
=(\bfx_t,\bfy_t,\bfz_t, B_t)
\mid t\in \bbT_n 
\}$ of degree $R$
with coefficients in a semifield $\bbP$ }
is
defined in the same way as the ordinary one
in Definition \ref{defn:clusterpattern1}
by replacing seeds of CA therein with seeds of degree $R$ of GCA.

\begin{ex}
(1).
A cluster pattern  ${\mathbf{\Sigma}}_R=\{ {\Sigma}_t
=({\bfx}_t,{\bfy}_t,{\bfz}_t,B_t)
\mid t\in \bbT_n 
\}$  of degree $R$ with coefficients in $\mathbb{P}
=\mathbb{Q}_{\mathrm{sf}}({\bfy}_{{t}_0},{\bfz}_{t_0})
$
 for some $t_0\in \mathbb{T}_n$
is called a 
{\em cluster pattern of degree $R$ with universal coefficients}.
\par

(2).
A cluster pattern  $\tilde{\mathbf{\Sigma}}_R[{t}_0]=\{ \tilde{\Sigma}_t
=(\tilde{\bfx}_t,\tilde{\bfy}_t,\tilde{\bfz}_t,B_t)
\mid t\in \bbT_n 
\}$  of degree $R$ with coefficients in $\mathbb{P}
=\mathrm{Trop}(\tilde{\bfy}_{{t}_0},\tilde{\bfz}_{t_0})
$
 for some $t_0\in \mathbb{T}_n$
is called a 
{\em cluster pattern of degree $R$ with principal coefficients at $t_0$}.

\par
(3).
A cluster pattern  $\check{\mathbf{\Sigma}}_R[{t}_0]=\{ \check{\Sigma}_t
=(\check{\bfx}_t,\check{\bfy}_t,\check{\bfz}_t,B_t)
\mid t\in \bbT_n 
\}$  of degree $R$ with coefficients in $\mathbb{P}
=\mathrm{Trop}(\check{\bfy}_{{t}_0})
$
 for some $t_0\in \mathbb{T}_n$
 such that $\check{z}_{i,s;t}=1$ for any $i,s$, and $t$
is called a 
{\em cluster pattern of degree $R$ with $Y$-principal coefficients at $t_0$}.
In this case one can set 
$ \check{{\Sigma}}_t
=(\check{\bfx}_t,\check{\bfy}_t,B_t)$.

\par
(4).
A cluster pattern  $\underline{\mathbf{\Sigma}}_R=\{ \underline{\Sigma}_t
=(\underline{\bfx}_t,\underline{\bfy}_t,\underline{\bfz}_t,B_t)
\mid t\in \bbT_n 
\}$  of degree $R$ with coefficients in the trivial semifield $\mathbf{1}$
is called a 
{\em cluster pattern of degree $R$ without  coefficients}.
In this case one can set 
$ \underline{\Sigma}_t
=(\underline{\bfx}_t,B_t)$.

\end{ex}

\subsection{$FGC$-patterns of GCA and separation formulas}

Next we introduce $FGC$-patterns of GCA following \cite{Nakanishi14a}.

Let  $\mathbf{\Sigma}_R=\{ \Sigma_t
=(\bfx_t,\bfy_t,\bfz_t,B_t)
\mid t\in \bbT_n 
\}$ be a cluster pattern of degree $R$
with coefficients in any semifield $\bbP$.
We arbitrary choose a distinguished
point (the {\em initial point\/}) $t_0$ in $\bbT_n$.
The seed $(\bfx_{t_0},\bfy_{t_0},\bfz_{t_0}, B_{t_0})$ at $t_0$ is called
the {\em initial seed} of $\mathbf{\Sigma}_R$.
We use the following simplified notation
\begin{align}
\bfz_{t_0}=\bfz=(z_{i,s})_{1\leq i \leq n, 1\leq s \leq r_i-1},
\end{align}
together with the notation \eqref{eq:initial2}.

Let us extract a family 
$\mathbf{B}_R=\{B_t)
\mid t\in \bbT_n 
\}$ from $\mathbf{\Sigma}_R$,
which we call a {\em $B$-pattern of degree $R$}.
For any  $B$-pattern 
$
\mathbf{B}_R=\{B_t)
\mid t\in \bbT_n 
\}$
of degree $R$
and any $t_0\in \mathbb{T}_n$,
 we introduce
a family of  quintuplet $\mathbf{\Gamma}(\mathbf{B}_R, t_0)
=
\{ \Gamma_t=(\mathbf{F}_{t}(\bfu,\bfv), G_t, C_t, \bfv_t,B_t)
\mid t\in \mathbb{T}_n \}$
uniquely determined 
from $\mathbf{B}_R$ and $t_0$
as follows.
For each $t\in \mathbb{T}_n$,
$\Gamma_t=(\mathbf{F}_{t}(\bfu,\bfv), G_t, C_t, \bfv_t,B_t)$
is a quintuplet (let us call it an {\em $FGC$-seed of degree $R$}),
where
\begin{itemize}
\item
$\mathbf{F}_{t}(\bfu,\bfv)=(F_{1;t}(\bfu,\bfv),\dots, F_{n;t}(\bfu,\bfv))$ is an $n$-tuple
of  rational functions (called  {\em $F$-poly\-nomials}) in
variables $\bfu=(u_1,\dots,u_n)$
and $\bfv=(v_{i,s})_{1\leq i \leq n, 1\leq s \leq r_i-1}$
with coefficients in $\bbQ$
having a subtraction-free expression,
i.e., $F_{i;t}(\bfu,\bfv)\in \bbQ_{\mathrm{sf}}(\bfu,\bfv)$,
\item
$G_t$ and $C_t$ are $n\times n$ integer matrices
(called a {\em $G$-matrix} and a {\em $C$-matrix}),
\item
and $\bfv_t=(v_{i,s;t})_{1\leq i \leq n, 1\leq s \leq r_i-1}$
is obtained from $\bfv$ by some permutations.
\end{itemize}
At the initial point $t_0$, they are given by
\begin{align}
\label{eq:ginitial1}
{F}_{i;t_0}(\bfu,\bfv)=1 \quad (i=1,\dots,n),
\quad
G_{t_0}=C_{t_0}=I,
\quad
\bfv_{t_0}=\bfv.
\end{align}
For $t,t'\in \mathbb{T}_n$
which  are connected by an edge labeled by $k$,
$\Gamma_t$ and $\Gamma_{t'}$ are related by the following {\em mutation
at $k$}:
\begingroup
\allowdisplaybreaks
\begin{align}
\label{eq:gFmut1}
 F_{i;t'}(\bfu,\bfv)&=
 \begin{cases}
\frac
 {
  \displaystyle
 M_{k;t}(\bfu,\bfv)
 }
{ \displaystyle
 F_{k;t}(\bfu,\bfv)
 }
   &
 i= k
 \\
 F_{i;t}(\bfu,\bfv)  &
 i \neq k,
 \end{cases}
 \\
 \label{eq:ggmut1}
 g_{ij}^{t'}&=
 \begin{cases}
 \displaystyle
 -g_{ik}^t
 + 
  r_k
  \left(
  \sum_{\ell=1}^n g_{i\ell}^t [-b_{\ell k}^t]_+
 -  \sum_{\ell=1}^nb_{i\ell}  [-c_{\ell k}^t]_+ 
 \right)
 &
 j= k
 \\
 g_{ij}^t  &
 j \neq k,
 \end{cases}
 \\
 \label{eq:gcmut1}
 c_{ij}^{t'}&=
 \begin{cases}
 -c_{ik}^t
 &
 j= k,
 \\
 c_{ij}^t + 
 r_k\left(
 c_{ik}^t [b_{kj}^t]_+
 + [-c_{ik}^t]_+  b_{kj}^t
 \right)
 &
 j \neq k,
 \end{cases}
 \\
 \label{eq:gvmut1}
v_{i,s;t'}
&=
\begin{cases}
 v_{k,r_k-s;t}
& i = k,
\\
v_{i,s;t}
& i \neq k,
\end{cases}
\end{align}
where
\begin{align}
\label{eq:gM1}
 M_{k;t}(\bfu,\bfv)
 &=
   \left(
      \prod_{j=1}^{n}
  u_j^{[-c_{jk}^t]_+}
  F_{j;t}(\bfu,\bfv)^{[-b_{jk}^t]_+}
  \right)^{r_k}
\sum_{s=0}^{r_k}
v_{k,s;t}
\left(
    \prod_{j=1}^{n}
    u_j^{c_{jk}^t}
  F_{j;t}(\bfu,\bfv)^{b_{jk}^t}
  \right)^s
  \end{align}
\endgroup
with $v_{k,0;t}=v_{k,r_k;t}=1$.
Note that $v$-variables mutate as $z$-variables in \eqref{eq:gzmut1}.
Again, these mutations are involutive.
We call $\mathbf{\Gamma}(\mathbf{B}_R,t_0)$ the {\em $FGC$-pattern
of $\mathbf{B}_R$ with the initial point $t_0$}.

Here is the Laurent property of GCA.

\begin{thm}[{\cite[Theorem 2.5]{Chekhov11},\cite[Proposition 3.3]{Nakanishi14a}}]
\label{thm:gFpoly1}
For any $i=1,\dots,n$ and $t\in \bbT_n$,
the function $F_{i;t}(\bfu,\bfv)$ is a polynomial in $\bfu$ and $\bfv$
with coefficients in $\bbZ$.
\end{thm}

We also note an analogue of
a well known formula for CA (Cf. \cite[Proposition 5.2]{Fomin07}).

\begin{prop}
\label{prop:gFtrop1}
The following formula holds for any $i$ and $t\in \mathbb{T}_n$:
\begin{align}
\label{eq:gFtrop1}
F_{i,t}\vert_{\mathrm{Trop(\bfy,\bfz)}}(\bfy,\bfz)=1.
\end{align}
\end{prop}
\begin{proof}
This can be easily deduced from
\eqref{eq:gFmut1} and \eqref{eq:gM1} by induction on $t\in \mathbb{T}_n$.
\end{proof}

There are two kinds of
analogues of Proposition \ref{prop:gtrop1},
both of which are  useful.
The first one is for cluster patterns with principal coefficients.

\begin{prop}[{\cite[Definition 3.7 \& Proposition 3.8]{Nakanishi14a}}]
\label{prop:gtrop1}
Let $\tilde{\mathbf{\Sigma}}_R[t_0]=\{ \tilde{\Sigma}_t
=(\tilde{\bfx}_t,\tilde{\bfy}_t,\tilde{\bfz}_t,B_t)
\mid t\in \bbT_n 
\}$
be a cluster pattern of degree $R$ with principal coefficients at  any point $t_0$,
and 
let
$\mathbf{\Gamma}(\mathbf{B}_R,t_0)
=
\{ \Gamma_t=(\mathbf{F}_{t}(\bfu,\bfv), G_t, C_t, \allowbreak
{\bfv}_t,B_t)
\mid t\in \mathbb{T}_n \}$
be the $FGC$-pattern  of $\mathbf{B}_R$ with
 the initial point $t_0$.
 Then, the following formula holds for any $t\in \mathbb{T}_n$:
\begin{align}
\label{eq:gpr1}
\tilde{y}_{i;t}&=\prod_{j=1}^n \tilde{y}_j ^{c^t_{ji}}.
\end{align}
\end{prop}
\begin{proof} This is proved by comparing the mutations
\eqref{eq:gymut1} and \eqref{eq:gcmut1}.
\end{proof}

The second one is for cluster patterns with $Y$-principal coefficients.
\begin{prop}
\label{prop:gtrop2}
Let $\check{\mathbf{\Sigma}}_R[t_0]=\{ \check{\Sigma}_t
=(\check{\bfx}_t,\check{\bfy}_t,B_t)
\mid t\in \bbT_n 
\}$
be a cluster pattern of degree $R$ with $Y$-principal coefficients at  any point $t_0$,
and 
let
$\mathbf{\Gamma}(\mathbf{B}_R,t_0)
=
\{ \Gamma_t=(\mathbf{F}_{t}(\bfu,\bfv), G_t, C_t, \allowbreak
{\bfv}_t,B_t)
\mid t\in \mathbb{T}_n \}$
be the $FGC$-pattern  of $\mathbf{B}_R$ with
 the initial point $t_0$.
 Then, the following formula holds for any $t\in \mathbb{T}_n$:
\begin{align}
\label{eq:gpr2}
\check{y}_{i;t}&=\prod_{j=1}^n \check{y}_j ^{c^t_{ji}}.
\end{align}
\end{prop}
\begin{proof} Again, this is proved by comparing the mutations
\eqref{eq:gymut1} and \eqref{eq:gcmut1}.
\end{proof}

The separation formulas for GCA are just in the same form as the ones for CA.

\begin{thm}[{Separation Formulas \cite[Theorems 3.22 \& 3.23]{Nakanishi14a}}]
\label{thm:gsep1}
For  a cluster pattern 
 $\mathbf{\Sigma}_R=\{ \Sigma_t
=(\bfx_t,\bfy_t,\bfz_t, B_t)
\mid t\in \bbT_n 
\}$
of degree $R$
with coefficients in any semifield $\bbP$,
the following formulas hold
for any $t\in \mathbb{T}_n$:
\begin{align}
\label{eq:gsep1}
x_{i;t}&=
\left(
\prod_{j=1}^n
x_j^{g_{ji}^t}
\right)
\frac{F_{i;t}(\hat{\bfy},\bfz)}{F_{i;t}\vert_{\bbP}(\bfy,\bfz)},
\quad
\hat{y}_i=y_i \prod_{j=1}^n x_j^{b_{ji}}
\in \mathcal{F},
\\
\label{eq:gsep2}
y_{i;t}&=
\left(
\prod_{j=1}^n
y_j^{c_{ji}^t}
\right)
\prod_{j=1}^n
F_{j;t}\vert_{\bbP}(\bfy,\bfz)^{b_{ji}^t}.
\end{align}
\end{thm}
\subsection{Companion patterns}

For any $FGC$-pattern of GCA,
we associate a pair of ``companion" $FGC$-patterns of CA
based on the results in
 \cite{Nakanishi14a,Nakanishi15}.
 
 As remarked after Definition \ref{defn:gmutation1},
 for any 
 $B$-pattern
$\mathbf{B}_R=\{B_t)
\mid t\in \bbT_n 
\}$
of degree $R$,
we can associate a  pair of $B$-patterns (of CA),
 \begin{align}
R\mathbf{B}&=\{ RB_t \mid t\in \mathbb{T}_n\},
\\
\mathbf{B}R&=\{ B_t R \mid t\in \mathbb{T}_n\},
\end{align}
which form a conjugate pair in the sense of Section
\ref{subsec:conjugate1}.
(Here, $B_t$ is an integer matrix.)

\begin{defn}[Companion patterns]
\label{defn:companion1}
Let
$\mathbf{\Gamma}(\mathbf{B}_R,t_0)
=
\{ \Gamma_t=(\mathbf{F}_{t}(\bfu,\bfv),\allowbreak G_t, C_t, \bfv_t,B_t)
\mid t\in \mathbb{T}_n \}$
be the $FGC$-pattern  of $\mathbf{B}_R$ with
 any initial point $t_0$.
The  $FGC$-patterns (of CA),
\begin{align}
{}^L\mathbf{\Gamma}=\mathbf{\Gamma}(R\mathbf{B},t_0)
&=
\{ {}^L\Gamma_t=({}^L\mathbf{F}_{t}(\bfu), {}^LG_t, {}^LC_t, RB_t)
\mid t\in \mathbb{T}_n \}
\\
{}^R\mathbf{\Gamma}=\mathbf{\Gamma}(\mathbf{B}R,t_0)
&=
\{ {}^R\Gamma_t=({}^R\mathbf{F}_{t}(\bfu), {}^RG_t, {}^RC_t, B_tR)
\mid t\in \mathbb{T}_n \},
\end{align}
are called the {\em left- and right-companions of
$\mathbf{\Gamma}(\mathbf{B}_R,t_0)$}, respectively.
\end{defn}

They are companions in the following sense.

\begin{thm}
\label{thm:companion1}
The following equalities hold.
\par
(1) \cite[Propositions 3.9 \& 3.16]{Nakanishi14a}.
\begin{alignat}{3}
\label{eq:gc1}
G_t& =R^{-1}( {}^LG_t)R,
&\quad C_t&=  {}^LC_t,
\\
\label{eq:gg1}
G_t& = {}^RG_t,
&\quad C_t&=  R({}^RC_t)R^{-1}.
\end{alignat}
\par
(2) \cite[Propositions 4.3 \& 4.6]{Nakanishi15}.
\begin{align}
\label{eq:gff2}
F_{i;t}(\bfu,\bfv^{\mathrm{bin}})& = {}^LF_{i;t}(\bfu)^{r_i},
\quad
v^{\mathrm{bin}}_{i,s}=\binom{r_i}{s},
\\
\label{eq:gff1}
F_{i;t}(\bfu,\bf{0})& = {}^RF_{i;t}(u_1^{r_1},\dots,u_n^{r_n}).
\end{align}
\end{thm}
\begin{proof}
They are obtained by comparing the mutations
\eqref{eq:gFmut1}--\eqref{eq:gM1} for GCA
and 
the mutations
\eqref{eq:Fmut1}--\eqref{eq:M1} for CA.
\end{proof}

\section{Further fundamental results in GCA}

\subsection{Sign-coherence}
Let us present analogous results related to the sign-coherence
in Section \ref{subsec:sign-coherence}.
Let
$\mathbf{\Gamma}(\mathbf{B}_R,t_0)
=
\{ \Gamma_t=(\mathbf{F}_{t}(\bfu,\bfv), \allowbreak
G_t, C_t, \bfv_t,B_t)
\mid t\in \mathbb{T}_n \}$
be the $FGC$-pattern  of $\mathbf{B}_R$ with
 any initial point $t_0$.

The sign-coherence of $C$-matrices of GCA immediately follows from
 \eqref{eq:gc1} and Theorem \ref{thm:sign}.

\begin{thm}[{Sign-coherence, \cite[Theorem 3.20]{Nakanishi14a}}]
\label{thm:gsign}
Each $C$-matrix $C_t$ is column-sign coherent.
Namely, each column vector ($c$-vector) of $C_t$   is nonzero vector, and 
its components are either all nonnegative or all nonpositive.
\end{thm}

Other related results in Section \ref{subsec:sign-coherence}
also hold via
Theorem \ref{thm:companion1} and/or \ref{thm:gsign}.

 \begin{thm}
\cite[Theorem 3.19]{Nakanishi14a}
 \label{thm:gconst1}
 The constant term of each $F$-polynomial $F_{i;t}(\bfu,\bfv)$ is 1.
 Moreover, 
 $F_{i;t}(\mathbf{0},\bfv)=1$.
 \end{thm}
\begin{proof}
Observe that if all $F_{i;t}(\bfu,\bfv)$ ($i=1,\dots,n$) satisfy 
this property for some $t$, then $M_{k;t}(\bfu,\bfv)$ in \eqref{eq:gM1}
also satisfy the same property thanks to
 the sign-coherence in Theorem \ref{thm:gsign}.
Then, the claim is proved by  induction on $t\in \mathbb{T}_n$
using \eqref{eq:gFmut1}.
\end{proof}

Let $D$ be a  common skew-symmetrizer of $R\mathbf{B}$.
\begin{thm}[{\cite[Proposition 3.21]{Nakanishi14a}}]
\label{thm:gdual1}
The following equalities hold for any $t\in \mathbb{T}_n$:
\begin{align}
\label{eq:gdual1}
\text{(Duality)}
\quad
D^{-1} R^{-1}G_t^T RD C_t &= I,
\\
\label{eq:gbc1}
DRB_t  &= C_t^T (DRB) C_t,
\\
\label{eq:gdet1}
|\det G_t|& =|\det C_t|=1.
\end{align}
\end{thm}
\begin{proof}
This is obtained from Theorems \ref{thm:dual1} and 
\ref{thm:companion1}.
\end{proof}

Define $\sigma G_t$, $\sigma C_t$, $\sigma B_t$ as before
in \eqref{eq:s2}, \eqref{eq:s3}, \eqref{eq:s6},
respectively.

\begin{prop}
\label{prop:gcompat1}
The following properties hold.
\par
(1). If $C_{t_1} = \sigma C_{t_2}$ occurs for some $t_1,t_2\in \mathbb{T}_n$
and a permutation $\sigma\in S_n$,
then $\sigma$ is compatible with $D$ and $R$.
\par
(2). If $G_{t_1} = \sigma G_{t_2}$ occurs for some $t_1,t_2\in \mathbb{T}_n$
and a permutation $\sigma\in S_n$,
then $\sigma$ is compatible with $D$ and $R$.
\end{prop}
\begin{proof}
(1). The compatibility with $D$ is
 a consequence of Proposition \ref{prop:compat1}
and \eqref{eq:gc1}.
Also, 
by applying the  argument in the proof of
Proposition \ref{prop:compat1} to \eqref{eq:gdual1} intead of \eqref{eq:dual1},
the compatibility with $RD$ is shown.
Therefore,  $\sigma$ is also compatible with $R$.
The proof of (2) is similar.
\end{proof}

\begin{cor} 
\label{cor:gsigma1}
The following statements hold for $t_1,t_2\in \mathbb{T}_n$
and  a  permutation $\sigma\in S_n$,
where $P_{\sigma}$ is the permutation matrix in \eqref{eq:p1}.
\begin{gather}
\label{eq:gdual3}
C_{t_1} = P_{\sigma}  \quad \Longleftrightarrow\quad  G_{t_1} =P_{\sigma}
\quad \Longrightarrow
\quad
 B_{t_1} = \sigma B,
\\
\label{eq:gdual2}
C_{t_1} = \sigma C_{t_2}\quad \Longleftrightarrow\quad  G_{t_1} = \sigma G_{t_2}
 \quad\Longrightarrow  \quad B_{t_1} = \sigma B_{t_2}.
 \end{gather}
\end{cor}
\begin{proof}
This follows from Theorem \ref{thm:gdual1} and Proposition \ref{prop:gcompat1}.
\end{proof}

\subsection{Laurent positivity of GCA}

Next let us consider  analogous results related to the Laurent positivity
in Section \ref{subsec:Laurent}.
Again, let
$\mathbf{\Gamma}(\mathbf{B}_R,t_0)
=
\{ \Gamma_t=(\mathbf{F}_{t}(\bfu,\bfv), \allowbreak
G_t, C_t, \bfv_t,B_t)
\mid t\in \mathbb{T}_n \}$
be the $FGC$-pattern  of $\mathbf{B}_R$ with
 any initial point $t_0$.

Unfortunately, the Laurent positivity of GCA itself is still an open problem.

\begin{conj}[{Laurent positivity}]
\label{conj:gpositivity}
Each $F$-polynomial $F_{i;t}(\bfu,\bfv)$ has
only positive coefficients.
\end{conj}

We note that, though limited, all examples of $F$-polynomials in \cite{Nakanishi14a,Nakanishi15} have this property.
To obtain the synchronicity results, we need to assume the validity of the conjecture
in some part. 
We will make it clear in each proposition,
whenever the conjecture is assumed.

Here is an analogous result to Lemma \ref{lem:Fcri1},
which we will use later.

Let $F_{i;t}(\bfu,\mathbf{1})$ denote the polynomial in $\bfu$
obtained from $F_{i;t}(\bfu,\bfv)$ by setting all $v_{i,s}$ to be $1$.

\begin{lem}
\label{lem:gFcri1}
Assume that Conjecture \ref{conj:gpositivity} is true.
Let $\underline{\mathbf{\Sigma}}_R=\{ \underline{\Sigma}_t
=(\underline{\bfx}_t,B_t)
\mid t\in \bbT_n 
\}$
be a cluster pattern of degree $R$ without coefficients,
and let $\underline{\hat{y}}_{i;t}$ be the one \eqref{eq:gyhat1} for $\underline{\mathbf{\Sigma}}$, namely,
\begin{align}
\label{eq:gyhat3}
\underline{\hat{y}}_{i;t}=\prod_{j=1}^n\underline{x}_{j:t}^{b^t_{ji}}.
\end{align}
Then, under the specialization $\underline{x}_{i:t_1}=1$ ($i=1,\dots,n$) at any point $t_1$,
we have the following inequality for any  $t,t'\in \mathbb{T}_n$ and $i=1,\dots,n$:
\begin{align}
F_{i;t}(\underline{\hat{\mathbf{y}}}_{t'},\mathbf{1})\vert_{\underline{x}_{1:t_1}=\dots
=\underline{x}_{n:t_1}=1}\geq 1.
\end{align}
Moreover, the equality holds if and only if 
$F_{i;t}({\mathbf{u},\bfv})=1$.
\end{lem}
\begin{proof}
Under the assumption of Conjecture \ref{conj:gpositivity},
the polynomial  $F_{i;t}(\bfu,\mathbf{1})$
in $\bfu$ has only positive coefficients.
Furthermore, by  Theorem \ref{thm:gconst1},
the constant term of $F_{i;t}(\bfu,\mathbf{1})$ is 1.
Then, the rest of the the proof is the same as the proof of Lemma \ref{lem:Fcri1}.
\end{proof}

An analogue 
of  Theorem \ref{thm:C-F},
which is essential in our study,
can be obtained
by the same argument as before
using Lemma \ref{lem:gFcri1},
 where 
Conjecture \ref{conj:gpositivity} is assumed. 
However, fortunately one can also deduce it
through Theorem \ref{thm:C-F}
without assuming Conjecture \ref{conj:gpositivity}.
Here, we present the latter proof.

We use the following property of $F$-polynomials.

\begin{prop}[{Cf. \cite[Proposition 5.3, Conjecture 5.5]{Fomin07}}]
\label{prop:gF1}
(1).
Each $F$-polynomial $F_{i;t}(\bfu,\bfv)$ has a unique monomial of maximal degree
in $\bfu$.
Furthermore, the coefficient of this monomial is 1.
\par
(2).
An $F$-polynomial $F_{i;t}(\bfu,\bfv)$ is 1
if $F_{i;t}(\bfu,\mathbf{0})$ is 1,
or equivalently,
its right companion ${}^RF_{i;t}(\bfu)$ is $1$.
\end{prop}
\begin{proof}
(1).
This is proved in the same way as Theorem \ref{thm:gconst1}.
\par
(2). This is an immediate consequence of (1) and \eqref{eq:gff1}.
\end{proof}

Here is an analogue of Theorem \ref{thm:C-F}
with a proof {\em which does not rely on Conjecture \ref{conj:gpositivity}}.

\begin{thm}[Detropicalization]
\label{thm:gC-F}
The following statements hold for $t_1,t_2\in \mathbb{T}_n$
and  a  permutation $\sigma\in S_n$,
where $P_{\sigma}$ is the permutation matrix in \eqref{eq:p1}:
\begin{align}
\label{eq:ggf1}
G_{t_1} = P_{\sigma} \quad &\Longrightarrow\quad 
F_{i;t_1}(\bfu,\bfv)=1,
\\
\label{eq:ggf2}
G_{t_1} = \sigma G_{t_2}\quad &\Longrightarrow\quad 
{F}_{i;t_1}(\bfu,\bfv)
= {F}_{\sigma^{-1}(i);t_2}(\bfu,\bfv)\vert_{v_{j,s}\leftarrow
v_{\sigma(j),s;t_3}}.
\end{align}
Here, if  $t_2$ is connected to the initial point $t_0$ 
in $\mathbb{T}_n$
as
\begin{align}
\label{eq:gp2}
t _0
\overunder{k_1}{} 
\cdots
\overunder{k_p}{} 
t_2,
\end{align}
then $t_3\in \mathbb{T}_n$ is the one such that
\begin{align}
\label{eq:gp3}
t_3
\overunder{\sigma(k_1)}{} 
\cdots
\overunder{\sigma(k_p)}{} 
t_1.
\end{align}
Also, $v_{j,s}\leftarrow
v_{\sigma(j),s;t_3}$ means the replacement of $v_{j,s}$
by  $v_{\sigma(j),s;t_3}$ for all $j$ in a polynomial in $\bfv$.

\end{thm}

\begin{proof}
Proof of \eqref{eq:ggf1}.
Let $^{R}\mathbf{\Gamma}$ be the right companion of 
$\mathbf{\Gamma}(\mathbf{B}_R,t_0)$ in Definition
\ref{defn:companion1}.
Then,  by \eqref{eq:gg1}, we have $^{R}G_{t_1}=G_{t_1}=P_{\sigma}$.
Then,  by Theorem \ref{thm:C-F}, we have 
${}^RF_{i;t_1}(\bfu)=1$.
Therefore, by Proposition \ref{prop:gF1} (2),
we obtain ${F}_{i;t_1}(\bfu,\bfv)=1$.
 \par
Proof of \eqref{eq:ggf2}.
The proof is similar to the one for Theorem  \ref{thm:C-F}.
Let $t_3$ be the one in \eqref{eq:gp3}.
Repeating the same argument as before, we have
\begin{align}
(G_{t_3},C_{t_3}, B_{t_3})=
(P_{\sigma}, P_{\sigma},  \sigma B).
\end{align}
Thus, by \eqref{eq:ggf1}, we have 
$F_{i,t_3}(\bfu,\bfv)=1$ for any $i=1,\dots,n$.
Then, by comparing 
the mutations of $F$-polynomials \eqref{eq:gFmut1}
for  the sequences \eqref{eq:gp2} and \eqref{eq:gp3},
one can prove
 the claim \eqref{eq:ggf2} by induction on mutations.
 See \eqref{eq:gFnext1} and \eqref{eq:gFnext2} for the formulas
 in the first step.
\end{proof}

\subsection{Strong compatibility}
The second result \eqref{eq:ggf2} in Theorem \ref{thm:gC-F}
is not simply a permutation of $F$-polynomials.
unlike the corresponding one \eqref{eq:gf2}   in the CA case,
Let us study its implication  closely.
 {\em All results in this subsection do not rely on Conjecture \ref{conj:gpositivity}}.

Let $\sigma \bfx$ and $\sigma \bfy$ be as before
in \eqref{eq:s4} and \eqref{eq:s5}.

Let us start with an   ``obvious" result of Corollary \ref{cor:gsigma1}
and Theorem \ref{thm:gC-F}.

\begin{prop}
\label{prop:twist1}
Let
${{\mathbf{\Sigma}}}_R=\{ {\Sigma}_t
=({\bfx}_t,
{\bfy}_t,\bfz_t,B_t)
\mid t\in \bbT_n 
\}$
be a cluster pattern of degree $R$ with  coefficients in 
any  semifield $\mathbb{P}$,
and let
$\mathbf{\Gamma}(\mathbf{B}_R,t_0)=
\{ \Gamma_t=(\mathbf{F}_{t}(\bfu,\bfv), G_t, C_t, \bfv_t,B_t)
\mid t\in \mathbb{T}_n \}$
be the $FGC$-pattern of $\mathbf{B}_R$ with any initial point $t_0$.
Let $\bfx$, $\bfy$, $\bfz$ be the initial variables at $t_0$.
Then, 
for $t_1,t_2\in \mathbb{T}_n$
and  a  permutation $\sigma\in S_n$,
the following hold.
\par
(1). If $G_{t_1} = P_{\sigma}$, then we have
\begin{align}
\label{eq:gxxyy1}
\bfx_{t_1}=\sigma \bfx,
\quad
\bfy_{t_1}=\sigma \bfy.
\end{align}
\par
(2). 
Let $x_{i;t}(\bfx,\bfy,\bfz)$ and $y_{i;t}(\bfx,\bfy,\bfz)$ be $x_{i;t}$ and $y_{i;t}$
as functions of 
 $\bfx$, $\bfy$, $\bfz$ through the separation formulas
\eqref{eq:gsep1} and \eqref{eq:gsep2}, respectively.
 If $G_{t_1} = \sigma G_{t_2}$, then we have
\begin{align}
\label{eq:gxtwist1}
x_{i;t_1}(\bfx,\bfy,\bfz)&=x_{\sigma^{-1}(i);t_2}(\bfx,\bfy,\bfz)\vert_{z_{j,s}\leftarrow
z_{\sigma(j),s;t_3}},
\\
\label{eq:gytwist1}
y_{i;t_1}(\bfx,\bfy,\bfz)&=y_{\sigma^{-1}(i);t_2}(\bfx,\bfy,\bfz)\vert_{z_{j,s}\leftarrow
z_{\sigma(j),s;t_3}},
\end{align}
where $t_3$ is the one in
\eqref{eq:gp3}.
\end{prop}
\begin{proof}
They are straightforward consequences of the separation formulas
in Theorem \ref{thm:gsep1}
with
Corollary \ref{cor:gsigma1}
and Theorem \ref{thm:gC-F}.
\end{proof}

What is less  obvious is that
Proposition \ref{prop:twist1} further implies the following  result.

\begin{prop}
\label{prop:twist2}
In the same situation as Proposition \ref{prop:twist1},
 if $G_{t_1} = \sigma G_{t_2}$, then we have
\begin{align}
\label{eq:gxxyy2}
\bfx_{t_1}=\sigma \bfx_{t_2},
\quad
\bfy_{t_1}=\sigma \bfy_{t_2}.
\end{align}
\par
\end{prop}
\begin{proof}
Let us apply
 \eqref{eq:gytwist1}
 for 
a cluster pattern  $\check{\mathbf{\Sigma}}_R[{t}_2]=\{ \check{\Sigma}_t
=(\check{\bfx}_t,\check{\bfy}_t,B_t)
\mid t\in \bbT_n 
\}$  of degree $R$ with $Y$-principal coefficients at $t_2$.
Since all $z$-variables are trivial,
we obtain
$\check{\bfy}_{t_1}=\sigma \check{\bfy}_{t_2}$
by \eqref{eq:gytwist1}.
Now, 
let 
$\mathbf{\Gamma}'(\mathbf{B}_R,t_2)=
\{ \Gamma'_t=(\mathbf{F}'_{t}(\bfu,\bfv), G'_t, C'_t, \allowbreak \bfv'_t,B_t)
\mid t\in \mathbb{T}_n \}$
be the $FGC$-pattern of the above $\mathbf{B}_R$ with the initial point $t_2$.
Then, by Proposition \ref{prop:gtrop2}, we have $C'_{t_1}=P_{\sigma}$.
Therefore, $G'_{t_1}=P_{\sigma}$ by \eqref{eq:gdual3}.
Then, by  applying \eqref{eq:gxxyy1} for $\Sigma_R$,
we obtain $\bfx_{t_1}=\sigma \bfx_{t_2}$ and 
$\bfy_{t_1}=\sigma \bfy_{t_2}$.
\end{proof}

The apparent discrepancy  between  Propositions \ref{prop:twist1}
and   \ref{prop:twist2} 
leads to a remarkable condition for $\sigma$ therein,
which is stronger  than the compatibility with $R$.

\begin{defn} For a permutation $\sigma\in S_n$ and a diagonal matrix
$R=\mathrm{diag}(r_1,\dots,r_n)$ with positive integer diagonal entries,
we say that $\sigma$ is {\em strongly compatible with $R$} if
\begin{align}
\sigma(i)=i\
\text{for any $i$ with $r_i\geq 2$.}
\end{align}
\end{defn}

\begin{prop}
\label{prop:gcomp2}
Let
$\mathbf{\Gamma}(\mathbf{B}_R,t_0)=
\{ \Gamma_t=(\mathbf{F}_{t}(\bfu,\bfv), G_t, C_t, \bfv_t,B_t)
\mid t\in \mathbb{T}_n \}$
be the $FGC$-pattern of $\mathbf{B}_R$ with any initial point $t_0$.
Suppose that
 $\sigma\in S_n$ and $t_1,t_2\in \mathbb{T}_n$ satisfy
the condition
 \begin{align}
 G_{t_1}=\sigma G_{t_2}.
 \end{align}
 Let $t_3\in \mathbb{T}_n$ be the one in
 Theorem \ref{thm:gC-F}.
 Then, the following properties hold.
 \par
 (1). The following equality holds for any $i$ with $r_i\geq2$ and  $s=1,\dots,r_i-1$:
\begin{align}
\label{eq:gvsigma1}
v_{i,s}=v_{\sigma(i),s;t_3}.
\end{align}
\par
(2). $\sigma$ is strongly compatible with $R$.
\par
(3).  The following equality holds for any $i$:
\begin{align}
\label{eq:ggf3}
{F}_{i;t_1}(\bfu,\bfv)
= {F}_{\sigma^{-1}(i);t_2}(\bfu,\bfv).
\end{align}
\par
(4). Suppose that $t_0$ and $t_3$ are connected in $ \mathbb{T}_n$ as
\begin{align}
\label{eq:gp22}
t _0
\overunder{\ell_1}{} 
\cdots
\overunder{\ell_m}{} 
t_3.
\end{align}
Then, for each $i$ with $r_i\geq3$, the number of the appearance of $i$
in the sequence $(\ell_1,\dots, \ell_m)$ is even.
\par
(5). Suppose that $t_1$ and $t_2$ are connected in $ \mathbb{T}_n$ as
\begin{align}
\label{eq:gp32}
t _1
\overunder{\ell'_1}{} 
\cdots
\overunder{\ell'_{s}}{} 
t_2.
\end{align}
Then, for each $i$ with $r_i\geq 3$, the number of the appearance of $i$
in the sequence $(\ell'_1,\dots, \ell'_s)$ is even.
\par
(6). The following equality holds:
\begin{align}
\label{eq:gvv1}
\bfv_{t_1}=\bfv_{t_2}.
\end{align}
\end{prop}
\begin{proof}
(1). Choose any $k=1,\dots,n$, and define $t'_0,t'_3\in \mathbb{T}_n$ as
\begin{align}
\label{eq:gp4}
t _0&
\overunder{k}{} 
t'_0,
\\
\label{eq:gp5}
t_3&
\overunder{\sigma(k)}{} 
t'_3.
\end{align}
Recall that $G_{t_0}=C_{t_0}=I$, $F_{i;t_0}(\bfu,\bfv)=1$,
and, as we see in the proof of \eqref{eq:ggf2},
$G_{t_3}=C_{t_3}=P_{\sigma}$, and $F_{i;t_3}(\bfu,\bfv)=1$. 
Therefore, by the mutation \eqref{eq:gFmut1},
we obtain
\begin{align}
\label{eq:gFnext1}
F_{k;t'_0}(\bfu,\bfv)&=
\sum_{s=0}^{r_k} v_{k,s}u_k^s,
\\
\label{eq:gFnext2}
F_{\sigma(k);t'_3}(\bfu,\bfv)&=
\sum_{s=0}^{r_k} v_{\sigma(k),s;t_3}u_k^s
.
\end{align}
(This confirms \eqref{eq:ggf2} in the simplest case.)
We also note that,
by \eqref{eq:comp2}, we have
\begin{align}
\label{eq:gGG1}
G_{t'_3}=\sigma G_{t'_0}.
\end{align}
Let  $\tilde{\mathbf{\Sigma}}_R[{t}_0]=\{ \tilde{\Sigma}_t
=(\tilde{\bfx}_t,\tilde{\bfy}_t,\tilde{\bfz}_t,B_t)
\mid t\in \bbT_n 
\}$  be a cluster pattern of degree $R$
with principal coefficients at $t_0$.
Then, by the separation formula \eqref{eq:gsep1},
Proposition \ref{prop:gFtrop1},
and \eqref{eq:gGG1},
we have
\begin{align}
\tilde{x}_{k;t'_0}&=
\left(
\prod_{j=1}^n
\tilde{x}_j^{g_{jk}^{t'_0}}
\right)
{F_{k;t'_0}(\hat{\tilde\bfy},\tilde\bfz)},
\quad
\hat{\tilde y}_i=\tilde{y}_i \prod_{j=1}^n \tilde{x}_j^{b_{ji}},
\\
\tilde{x}_{\sigma(k);t'_3}&=
\left(
\prod_{j=1}^n
\tilde{x}_j^{g_{j\sigma(k)}^{t'_3}}
\right)
{F_{\sigma(k);t'_3}(\hat{\tilde\bfy},\tilde\bfz)}
=
\left(
\prod_{j=1}^n
\tilde{x}_j^{g_{jk}^{t'_0}}
\right)
{F_{\sigma(k);t'_3}(\hat{\tilde\bfy},\tilde\bfz)}
.
\end{align}
Meanwhile,
by \eqref{eq:gGG1} and Theorem \ref{thm:gbasic1},
we have
$
\tilde{x}_{k;t'_0}=\tilde{x}_{\sigma(k);t'_3}.
$
Therefore, we obtain the identity
\begin{align}
{F_{k;t'_0}(\hat{\tilde\bfy},\tilde\bfz)}=
{F_{\sigma(k);t'_3}(\hat{\tilde\bfy},\tilde\bfz)}.
\end{align}
Setting $\tilde{x}_{i}=1$ ($i=1,\dots,n$) therein, we have the identity
\begin{align}
{F_{k;t'_0}({\tilde\bfy},\tilde\bfz)}=
{F_{\sigma(k);t'_3}({\tilde\bfy},\tilde\bfz)}.
\end{align}
Since the principal coefficients
$\tilde\bfy$ and $\tilde\bfz$ at $t_0$ are independent variables,
this is equivalent to the identity
\begin{align}
{F_{k;t'_0}({\bfu},\bfv)}=
{F_{\sigma(k);t'_3}({\bfu},\bfv)}.
\end{align}
Then, comparing 
\eqref{eq:gFnext1} and \eqref{eq:gFnext2},
we conclude that \eqref{eq:gvsigma1} holds.
\par
(2). Under the mutation rule \eqref{eq:gvmut1},
we have $v_{i,s;t}=v_{i,s}$ or $v_{i,r_i-s}$ for any $t$.
Therefore, \eqref{eq:gvsigma1} occurs only when
$\sigma$ is strongly compatible with $R$.
\par
(3).
This is a consequence of
 \eqref{eq:ggf2} and \eqref{eq:gvsigma1}.
\par
(4). 
Let $i$ be the one with $r_i\geq 3$.
During the sequence of mutations
\eqref{eq:gp22}, 
for fixed $i$ and $s$,
the variables $v_{i,s;t}$ alternate between $v_{i,s}$ and $v_{i;r_i-s}$
whenever $\mu_i$ occurs.
Therefore,  \eqref{eq:gvsigma1} implies that
$\mu_i$ occurs in even times.
\par
(5).
As a path between $t_1$ and $t_2$ in $\mathbb{T}_n$,
it is enough to consider the concatenation
of \eqref{eq:gp2}, \eqref{eq:gp3}, and \eqref{eq:gp22}.
For $i$ with $r_i\geq 2$, the numbers of the appearance of $\mu_i$
in \eqref{eq:gp2} and \eqref{eq:gp3} coincide thanks 
to the strong compatibility $\sigma(i)=i$ proved in (2).
Then, combining this fact with  (4), we obtain the claim.
\par
(6). We show $v_{i,s;t_1}=v_{i,s;t_2}$ for $i$ with $r_i \geq 2$.
For $i$ with $r_i\geq 3$, this is a corollary of  (5).
For $i$ with $r_i = 2$, this is always true.
\end{proof}

Let us formally define the action of a permutation $\sigma\in S_n$
on $z$-variables for a seed  $(\bfx,\bfy,\bfz,B)$ of degree $R$
in the same way as other $x$- and $y$-variables
as
\begin{align}
 \sigma \bfz=\bfz', 
 \quad
 z'_{i,s}=z_{\sigma^{-1}(i),s}.
 \end{align}
 If $\sigma$ is strongly compatible with $R$,
 then $ \sigma \bfz=\bfz$.
 Namely, the action is trivial.
 
We have the following supplementary result to Proposition
 \ref{prop:twist2}.
 \begin{prop}
\label{prop:twist3}
In the same situation as Proposition \ref{prop:twist1},
 if $G_{t_1} = \sigma G_{t_2}$, then we have
\begin{align}
\label{eq:gzz2}
\bfz_{t_1}=\bfz_{t_2}=\sigma \bfz_{t_2}.
\end{align}
\par
\end{prop}
\begin{proof}
Recall that  $v$-variables and  $z$-variables
mutate by the same rule \eqref{eq:gzmut1} and  \eqref{eq:gvmut1}.
Also, for each $t$, $\bfv_t$ is a tuple of independent formal variables.
Therefore,  by \eqref{eq:gvv1},
we have $\bfz_{t_1}=\bfz_{t_2}$.
Furthermore, since $\sigma$ is strongly compatible with $R$ by 
Proposition \ref{prop:gcomp2} (2),
we have $\bfz_{t_2}=\sigma \bfz_{t_2}$.
\end{proof}

\section{Synchronicity phenomenon in GCA}
\label{sec:synchronicity2}

\subsection{$xy/GC$-synchronicity}
We are ready to give  the $xy/GC$ synchronicity for GCA,
 which are analogous to the one in Section 
\ref{subsec:xyGC}.

\begin{thm}[$xy/GC$ synchronicity]
\label{thm:gbasic1}
Assume that Conjecture \ref{conj:gpositivity} is true.
Let
${{\mathbf{\Sigma}}}_R=\{ {\Sigma}_t
=({\bfx}_t,
{\bfy}_t,\bfz_t,B_t)
\mid t\in \bbT_n 
\}$
be a cluster pattern of degree $R$ with  coefficients in 
any  semifield $\mathbb{P}$,
and let
$\mathbf{\Gamma}(\mathbf{B}_R,t_0)=
\{ \Gamma_t=(\mathbf{F}_{t}(\bfu,\bfv), G_t, C_t, \bfv_t,\allowbreak B_t)
\mid t\in \mathbb{T}_n \}$
be the $FGC$-pattern of $\mathbf{B}_R$ with any initial point $t_0$.
Then, 
for $t_1,t_2\in \mathbb{T}_n$
and  a  permutation $\sigma\in S_n$,
 the following three conditions are equivalent:
\par
(a). $G_{t_1}=\sigma G_{t_2}$.
\par
(b). $C_{t_1}=\sigma C_{t_2}$.
\par
(c). $\bfx_{t_1}=\sigma \bfx_{t_2}$.
\par
\noindent
Moreover, one of Conditions (a)--(c) implies the following conditions:
\par
(d). $\bfy_{t_1}=\sigma \bfy_{t_2}$.
\par
(e). $\bfz_{t_1}=\bfz_{t_2}=\sigma \bfz_{t_2}$.
\end{thm}
\begin{proof}
((a) $\Longleftrightarrow$ (b)).
This was already stated in
\eqref{eq:gdual2}.
\par
 ((a) $\Longrightarrow$ (c), (d), (e)).
  This was proved in Propositions
  \ref{prop:twist2} and   \ref{prop:twist3}.
\par
((c) $\Longrightarrow$ (b)).
This is proved  in the  same manner as Theorem \ref{thm:basic1}
by using a cluster pattern of degree $R$ without coefficients,
where we replace
 Lemma
\ref{lem:Fcri1}
with
 Lemma
\ref{lem:gFcri1}.
(This is only where we need Conjecture \ref{conj:gpositivity}.)

\end{proof}

Let us introduce a condition for the implication
(d) $\Longrightarrow$ (a), (b).

\begin{defn}
\label{defn:gcover1}
Let 
${{\mathbf{\Sigma}}}_R=\{ {\Sigma}_t
=({\bfx}_t,
{\bfy},\bfz_t,B_t)
\mid t\in \bbT_n 
\}$
be a cluster pattern of degree $R$ with  coefficients in 
any  semifield $\mathbb{P}$,
and
$\check{\mathbf{\Sigma}}_R[t'_0]=\{ \check{\Sigma}_t
=(\check{\bfx}_t,\check{\bfy}_t,B_t)
\mid t\in \bbT_n 
\}$
 be a cluster pattern of degree $R$ with $Y$-principal
coefficients at  $t'_0$,
such that they share a common $B$-pattern $\mathbf{B}_R$
of degree $R$.
Let $\langle \bfy_{t'_0}, \bfz_{t'_0}\rangle$ be the subsemifield 
in  $\mathbb{P}$
generated by $ \bfy_{t'_0}$ and $\bfz_{t'_0}$.
Then, we say that {\em $\mathbf{\Sigma}_R$ covers 
$\check{\mathbf{\Sigma}}_R[t'_0]$} if there is a semifield homomorphism
such that
\begin{align}
\label{eq:gsemihom4}
\begin{matrix}
\varphi:&\langle \bfy_{t'_0}, \bfz_{t'_0}\rangle&\rightarrow 
&\mathrm{Trop}(\check{\bfy}_{t'_0})\\
&y_{i;t'_0}
&
\mapsto
&
\check{y}_{i;t'_0}
\\
&{z}_{i,s;t'_0}
&
\mapsto
&
\quad
1
\quad .
\end{matrix}
\end{align}
\end{defn}

Now we present the second half of the statement of the $xy/GC$ synchronicity,
which does not rely on Conjecture \ref{conj:gpositivity}.

\begin{thm}[$xy/GC$ synchronicity]
\label{thm:gbasic2}
Let
${{\mathbf{\Sigma}}}_R=\{ {\Sigma}_t
=({\bfx}_t,
{\bfy}_t,
\bfz_t,
\allowbreak
B_t)
\mid t\in \bbT_n 
\}$
be a cluster pattern  of degree $R$ with  coefficients in 
 any  semifield $\mathbb{P}$
which covers a cluster pattern
of degree $R$
with $Y$-principal coefficients at some $t'_0$.
Then, Condition (d) implies Conditions (a)--(c) in Theorem
\ref{thm:gbasic1}.
\end{thm}
\begin{proof}
The proof is the same as Theorem \ref{thm:basic2}.
\end{proof}

\begin{cor}
\label{cor:gcomd1}
Let $D$ be a common skew-symmetrizer of $\mathbf{B}$.
\par
(1). 
Assume that Conjecture \ref{conj:gpositivity} is true.
Let 
${{\mathbf{\Sigma}}}_R=\{ {\Sigma}_t
=({\bfx}_t,
{\bfy}_t,
\bfz_t,
B_t)
\mid t\in \bbT_n 
\}$
be a cluster pattern of degree $R$  with  coefficients in 
any  semifield $\mathbb{P}$.
If $\bfx_{t_1}=\sigma \bfx_{t_2}$ occurs for some $t_1,t_2\in \mathbb{T}_n$
 and a permutation $\sigma\in S_n$,
 then $\sigma$ is compatible with $D$ and strongly compatible with $R$.
 \par
 (2).
 Let
${{\mathbf{\Sigma}}}_R
=\{ {\Sigma}_t
=({\bfx}_t,
{\bfy}_t,
\bfz_t,
\allowbreak
B_t)
\mid t\in \bbT_n 
\}$
be a cluster pattern of $R$ with  coefficients in 
 any  semifield $\mathbb{P}$
which covers a cluster pattern of $R$
with $Y$-principal coefficients at some $t'_0$.
If $\bfy_{t_1}=\sigma \bfy_{t_2}$ occurs for some $t_1,t_2\in \mathbb{T}_n$
 and a permutation $\sigma\in S_n$,
 then $\sigma$ is compatible with $D$ and strongly compatible with $R$.
\end{cor}
\begin{proof}
This follows from Propositions \ref{prop:gcompat1},
\ref{prop:gcomp2}
and Theorems \ref{thm:gbasic1}, \ref{thm:gbasic2}.
\end{proof}

\begin{rem}
\label{ref:gcover1}
In parallel to Definition
\ref{defn:gcover1},
for 
${{\mathbf{\Sigma}}}_R$ therein
and
a cluster pattern 
$\tilde{\mathbf{\Sigma}}_R[t'_0]=\{ \tilde{\Sigma}_t
=(\tilde{\bfx}_t,\tilde{\bfy}_t,\tilde{\bfz}_t,B_t)
\mid t\in \bbT_n 
\}$
 of degree $R$ with principal
coefficients at  $t'_0$,
we say that {\em $\mathbf{\Sigma}_R$ covers 
$\tilde{\mathbf{\Sigma}}_R[t'_0]$} if there is a semifield homomorphism
such that
\begin{align}
\label{eq:gsemihom5}
\begin{matrix}
\varphi:&\langle \bfy_{t'_0}, \bfz_{t'_0}\rangle&\rightarrow 
&\mathrm{Trop}(\tilde{\bfy}_{t'_0},\tilde{\bfz}_{t'_0})\\
&y_{i;t'_0}
&
\mapsto
&
\tilde{y}_{i;t'_0}
\\
&
{z}_{i,s;t'_0}
&
\mapsto
&
\quad
\tilde{z}_{i,s;t'_0}
\quad .
\end{matrix}
\end{align}
In this case, one can prove Theorem \ref{thm:gbasic1} without 
assuming Conjecture \ref{conj:gpositivity}
by using a cluster pattern of degree $R$ with principal coefficients
and the reduction \eqref{eq:gff1}.
\end{rem}

\subsection{More synchronicity results}
\label{subsec:gconsequences}

For a seed of $\Sigma
=(\bfx,\bfy,\bfz, B)$ of degree $R$ and a permutation $\sigma\in S_n$,
we define the action of $\sigma$ on $\Sigma$
by
\begin{align}
\sigma \Sigma &=(\sigma\bfx,\sigma\bfy,\sigma\bfz, \sigma B).
\end{align}

Let  $\mathbf{\Sigma}_R=\{ \Sigma_t
=(\bfx_t,\bfy_t,\bfz_t,B_t)
\mid t\in \bbT_n 
\}$ be a cluster pattern of degree $R$
with coefficients in any semifield $\bbP$.
Suppose that $t_1,t_2 \in\mathbb{T}_n$
are connected in $\mathbb{T}_n$ as
\begin{align}
\label{eq:gseq1}
t _1
\overunder{k_1}{} 
\cdots
\overunder{k_p}{} 
t_2,
\end{align}
so that
\begin{align}
\label{eq:gseq2}
\Sigma_{t_2}=\mu_{k_p} \cdots \mu_{k_1}(\Sigma_{t_1}).
\end{align}

After observing Theorems \ref{thm:gbasic1} and \ref{thm:gbasic2},
the following definition is natural.

\begin{defn}[$\sigma$-periodicity]
We call a sequence of  mutations \eqref{eq:gseq2} of seeds 
of degree $R$
 a 
{\em $\sigma$-period\/} if
\begin{align}
\Sigma_{t_1} = \sigma\Sigma_{t_2}.
\end{align}
\end{defn}

Then, all results in 
Section 
\ref{subsec:consequences}
are naturally generalized as we present below.
The proofs are the same as before, so that we omit them.
Note that the assumption of Conjecture \ref{conj:gpositivity}
is necessary only for the statements involving $x$-variables.
See also Remark \ref{ref:gcover1} for the condition when
this assumption can be avoided.

\begin{thm}
\label{thm:gx1}
Assume that Conjecture \ref{conj:gpositivity} is true.
Let
${{\mathbf{\Sigma}}}_R=\{ {\Sigma}_t
=({\bfx}_t,
{\bfy}_t,
\bfz_t,
B_t)
\mid t\in \bbT_n 
\}$
be a cluster pattern of degree $R$ with  coefficients in 
any  semifield $\mathbb{P}$.
Then, the periodicity of $x$-variables $\bfx_t$  depends only on 
the $B$-pattern $\mathbf{B}_R$ of degree $R$  therein.
\end{thm}

\begin{thm}
\label{thm:gy1}
Let
${{\mathbf{\Sigma}}}_R=\{ {\Sigma}_t
=({\bfx}_t,
{\bfy}_t,
\bfz_t,
B_t)
\mid t\in \bbT_n 
\}$
be a cluster pattern of degree $R$ with  coefficients in 
any  semifield $\mathbb{P}$
such that
 $\mathbf{\Sigma}_R$ covers 
a cluster pattern of degree $R$ with $Y$-principal coefficients at 
some $t'_0$.
Then, the periodicity of $y$-variables $\bfy_t$  depends only on 
the $B$-pattern $\mathbf{B}_R$ of degree $R$ therein.
\end{thm}

\begin{thm}
\label{thm:gxy1}
Assume that Conjecture \ref{conj:gpositivity} is true.
The  $x$-variables in Theorem \ref{thm:gx1}
and the  $y$-variables in Theorem \ref{thm:gy1}
with a common $B$-pattern $\mathbf{B}_R$ of degree $R$
share the same periodicity.
\end{thm}

\begin{thm}
\label{thm:gseed1}
Assume that Conjecture \ref{conj:gpositivity} is true.
Let
${{\mathbf{\Sigma}}}_R=\{ {\Sigma}_t
=({\bfx}_t,
{\bfy}_t,
\bfz_t,
B_t)
\mid t\in \bbT_n 
\}$
be a cluster pattern of degree $R$ with  coefficients in 
any  semifield $\mathbb{P}$.
Then,   the   $x$-variables $\bfx_t$
and the  seeds $\Sigma_t$
share the same periodicity.
Furthermore, it depends only on 
the $B$-pattern $\mathbf{B}_R$  of degree $R$ therein.
\end{thm}

\begin{thm}
\label{thm:gyseed1}
Let
${{\mathbf{\Sigma}}}_R=\{ {\Sigma}_t
=({\bfx}_t,
{\bfy}_t,
\bfz_t,
B_t)
\mid t\in \bbT_n 
\}$
be a cluster pattern  of degree $R$ with  coefficients in 
any  semifield $\mathbb{P}$
such that $\mathbf{\Sigma}_R$ covers 
a cluster pattern  of degree $R$ with $Y$-principal coefficients at 
some $t'_0$.
Then, the $y$-variables $\bfy_t$ 
and  the $Y$-seeds $(\bfy_t,B_t)$ 
share the same periodicity.
Furthermore, it depends only on 
the $B$-pattern $\mathbf{B}_R$ of degree $R$ therein.
\end{thm}

\begin{thm}
\label{thm:gxyseed1}
Assume that Conjecture \ref{conj:gpositivity} is true.
Let
${{\mathbf{\Sigma}}}=\{ {\Sigma}_t
=({\bfx}_t,
{\bfy}_t,
\bfz_t,
B_t)
\mid t\in \bbT_n 
\}$
be a cluster pattern of degree $R$ with  coefficients in 
 any semifield $\mathbb{P}$
such that $\mathbf{\Sigma}$ covers 
a cluster pattern of degree $R$ with $Y$-principal coefficients at 
some $t'_0$.
Then, the seeds $\Sigma_t$ and the $Y$-seeds
$(\bfy_t,B_t)$ share the same  periodicity.
Furthermore, it depends only on
the $B$-pattern $\mathbf{B}_R$ of degree $R$ therein.
\end{thm}

\subsection{Synchronicity with companion patterns}
Another main result in Part \ref{part:GCA} is the following one.
See also Theorem \ref{thm:conj1}.

\begin{thm}
\label{thm:gcompanion1}
Assume that Conjecture \ref{conj:gpositivity} is true.
Let
${{\mathbf{\Sigma}}}_R=\{ {\Sigma}_t
=({\bfx}_t,
{\bfy}_t,
\bfz_t,
B_t)
\mid t\in \bbT_n 
\}$
be a cluster pattern of degree $R$ with  coefficients in 
any  semifield $\mathbb{P}$.
Let 
${}^L{{\mathbf{\Sigma}}}=\{ {}^L{\Sigma}_t
=({}^L{\bfx}_t,
{}^L{\bfy}_t,
RB_t)
\mid t\in \bbT_n 
\}$
and
${}^R{{\mathbf{\Sigma}}}=\{ {}^R{\Sigma}_t
=({}^R{\bfx}_t,
{}^R{\bfy}_t,
B_tR)
\mid t\in \bbT_n 
\}$
be cluster patterns with coefficients in any semifields
${}^L\mathbb{P}$ and ${}^R\mathbb{P}$,
respectively.
Then, the  periodicity of seeds $ {\Sigma}_t$
coincides with the common periodicity of 
seeds $ {}^L{\Sigma}_t$ and  ${}^R{\Sigma}_t$.
\end{thm}
\begin{proof}
This is a corollary of
Theorem \ref{thm:companion1}
together with
Theorems \ref{thm:basic1}, \ref{thm:seed1},
\ref{thm:gbasic1}, and \ref{thm:gseed1}.
\end{proof}

In other words, the seeds of a GCA
and the seeds of
its companion CA's
share the same exchange graph.

As a corollary, we have a strong version of
Theorem \ref{thm:conj1} (1) when $B$ is an integer matrix.
(We do not now if the statement still holds for a non-integer matrix $B$.)

\begin{cor}
\label{cor:gconj1}
Let $R\mathbf{B}$ and $\mathbf{B}R$ be a conjugate pair
such that $B$ is an integer (skew-symmetrizable) matrix.
Let
${{\mathbf{\Sigma}}}=\{ {\Sigma}_t
=({\bfx}_t,
{\bfy}_t,RB_t)
\mid t\in \bbT_n 
\}$
and
${{\mathbf{\Sigma}}}'=\{ {\Sigma}'_t
=({\bfx}'_t,
{\bfy}'_t,B_tR)
\mid t\in \bbT_n 
\}$
be  cluster patterns with coefficients
in  any semifields $\mathbb{P}$ and $\mathbb{P}'$, respectively.
If $ \Sigma_{t_1}=\sigma \Sigma_{t_2}$ or $ \Sigma'_{t_1}=\sigma \Sigma'_{t_2}$
 for some $t_1,t_2\in \mathbb{T}_n$
and  a permutation $\sigma\in S_n$,
then $\sigma$ is strongly compatible with $R$.
\end{cor}

\subsection{Synchronicity under dualities}
For completeness, we also give  analogous results  to
the ones in Section \ref{subsec:dualities}.

\begin{thm}
\label{thm:gdual4}
Assume that Conjecture \ref{conj:gpositivity} is true.
Let
${{\mathbf{\Sigma}}}_R=\{ {\Sigma}_t
=({\bfx}_t,
{\bfy}_t,
\bfz_t,
B_t)
\mid t\in \bbT_n 
\}$
and
${{\mathbf{\Sigma}}}'_R=\{ {\Sigma}'_t
=({\bfx}'_t,
{\bfy}'_t,
{\bfy}'_t,
B^T_t)
\mid t\in \bbT_n 
\}$
be  cluster patterns of degree $R$ with coefficients
in any semifields $\mathbb{P}$ and $\mathbb{P}'$, respectively.
Then, the seeds in ${\mathbf{\Sigma}}_R$ and the seeds in
${\mathbf{\Sigma}}'_R$ share the same  periodicity.
\end{thm}
\begin{proof}
$RB_t$ and $B^T_tR$ are related by the transposition duality
in Section \ref{subsec:dualities}.
Then, the claim follows from Theorems \ref{thm:gcompanion1}
and \ref{thm:dual4}.
\end{proof}

\begin{thm}
\label{thm:gdual5}
Assume that Conjecture \ref{conj:gpositivity} is true.
Let
${{\mathbf{\Sigma}}}_R=\{ {\Sigma}_t
=({\bfx}_t,
{\bfy}_t,
\bfz_t,
B_t)
\mid t\in \bbT_n 
\}$
and
${{\mathbf{\Sigma}}}'_R=\{ {\Sigma}'_t
=({\bfx}'_t,
{\bfy}'_t,
{\bfy}'_t,
-B_t)
\mid t\in \bbT_n 
\}$
be  cluster patterns of degree $R$ with coefficients
in any semifields $\mathbb{P}$ and $\mathbb{P}'$, respectively.
Then, the seeds in ${\mathbf{\Sigma}}_R$ and the seeds in
${\mathbf{\Sigma}}'_R$ share the same  periodicity.
\end{thm}
\begin{proof}
$RB_t$ and $-RB$ are related by the chiral duality
in Section \ref{subsec:dualities}.
Then, the claim follows from Theorems \ref{thm:gcompanion1}
and \ref{thm:dual5}.
\end{proof}

\begin{thm}
\label{thm:gdual6}
Assume that Conjecture \ref{conj:gpositivity} is true.
Let
${{\mathbf{\Sigma}}}_R=\{ {\Sigma}_t
=({\bfx}_t,
{\bfy}_t,
\bfz_t,
B_t)
\mid t\in \bbT_n 
\}$
and
${{\mathbf{\Sigma}}}'_R=\{ {\Sigma}'_t
=({\bfx}'_t,
{\bfy}'_t,
{\bfy}'_t,
-B^T_t)
\mid t\in \bbT_n 
\}$
be  cluster patterns of degree $R$ with coefficients
in any semifields $\mathbb{P}$ and $\mathbb{P}'$, respectively.
Then, the seeds in ${\mathbf{\Sigma}}_R$ and the seeds in
${\mathbf{\Sigma}}'_R$ share the same  periodicity.
\end{thm}
\begin{proof}
$RB_t$ and $-B^T_tR$ are related by the  Langlands duality
in Section \ref{subsec:dualities}.
Then, the claim follows from Theorems \ref{thm:gcompanion1}
and \ref{thm:dual6}.
\end{proof}

\bibliography{../../biblist/biblist.bib}

\newcommand{\etalchar}[1]{$^{#1}$}
\providecommand{\bysame}{\leavevmode\hbox to3em{\hrulefill}\thinspace}
\providecommand{\MR}{\relax\ifhmode\unskip\space\fi MR }
\providecommand{\MRhref}[2]{%
  \href{http://www.ams.org/mathscinet-getitem?mr=#1}{#2}
}
\providecommand{\href}[2]{#2}
\begin{thebibliography}{CIKLFP13}

\bibitem[BMRT07]{Buan05c}
A.~B. Buan, R.~J. Marsh, I.~Reiten, and G.~Todorov, \emph{Clusters and seeds in
  acyclic cluster algebras}, Proc. Amer. Math. Soc. \textbf{135} (2007),
  3049--3060; arXiv:math/0510359 [math.RT].

\bibitem[CHL18]{Cao17}
P.~Cao, M.~Huang, and F.~Li, \emph{A conjecture on {$C$}-matrices of cluster
  algebras}, Nagoya Math. J. \textbf{Published online} (2018), ;
  arXiv:1702.01221 [math.RA].

\bibitem[CIKLFP13]{Cerulli12}
G.~Cerulli~Irelli, B.~Keller, D.~Labardini-Fragoso, and P.~Plamondon,
  \emph{Liniear independence of cluster monomials for skew-symmetric cluster
  algebras}, Compos. Math. \textbf{149} (2013), 1753--1764; arXiv:1203.1307
  [math.RT].

\bibitem[CL18]{Cao18}
P.~Cao and F.~Li, \emph{The enough {$g$}-pairs property and denominator vectors
  of cluster algebras}, 2018, arXiv:1803.05281 [math.RT].

\bibitem[CS14]{Chekhov11}
L.~Chekhov and M.~Shapiro, \emph{Teichm\"uller spaces of {R}iemann surfaces
  with orbifold points of arbitrary order and cluster variables}, Int. Math.
  Res. Notices \textbf{2014} (2014), 2746--2772; arXiv:1111.3963 [math--ph].

\bibitem[Dav18]{Davison16}
B.~Davison, \emph{Positivity for quantum cluster algebras}, Ann. of Math.
  \textbf{187} (2018), 157--219; arXiv:1601.07918 [math.RT].

\bibitem[DWZ10]{Derksen10}
H.~Derksen, J.~Weyman, and A.~Zelevinsky, \emph{Quivers with potentials and
  their representations {II}: {A}pplications to cluster algebras}, J. Amer.
  Math. Soc. \textbf{23} (2010), 749--790; arXiv:0904.0676 [math.RA].

\bibitem[FG09a]{Fock03}
V.~V. Fock and A.~B. Goncharov, \emph{Cluster ensembles, quantization and the
  dilogarithm}, Annales Sci. de l'\'Ecole Norm. Sup. \textbf{42} (2009),
  865--930; arXiv:math/0311245 [math.AG].

\bibitem[FG09b]{Fock07}
\bysame, \emph{The quantum dilogarithm and representations of quantum cluster
  varieties}, Invent. Math. \textbf{172} (2009), 223--286; arXiv:math/0702397
  [math.QA].

\bibitem[FG19]{Fujiwara18}
S.~Fujiwara and Y.~Gyoda, \emph{Duality between final-seed and initail-seed
  mutations in cluster algebras}, SIGMA \textbf{15} (2019), 040, 24 pages;
  arXiv:1808.02156.

\bibitem[FST08]{Fomin08}
S.~Fomin, M.~Shapiro, and D.~Thurston, \emph{Cluster algebras and triangulated
  surfaces. {P}art {I}: {C}luster complexes}, Acta Math. \textbf{201} (2008),
  83--146; arXiv:math/0608367 [math.RA].

\bibitem[FZ02]{Fomin02}
S.~Fomin and A.~Zelevinsky, \emph{Cluster algebras {I}. {F}oundations}, J.
  Amer. Math. Soc. \textbf{15} (2002), 497--529 (electronic);
  arXiv:math/0104151 [math.RT].

\bibitem[FZ03a]{Fomin03a}
\bysame, \emph{Cluster algebras {II}. {F}inite type classification}, Invent.
  Math. \textbf{154} (2003), 63--121; arXiv:math/0208229 [math.RA].

\bibitem[FZ03b]{Fomin03b}
\bysame, \emph{Y-systems and generalized associahedra}, Ann. of Math.
  \textbf{158} (2003), 977--1018; arXiv:hep--th/0111053.

\bibitem[FZ07]{Fomin07}
\bysame, \emph{Cluster algebras {IV}. {C}oefficients}, Compositio Mathematica
  \textbf{143} (2007), 112--164; arXiv:math/0602259 [math.RT].

\bibitem[GHKK18]{Gross14}
M.~Gross, P.~Hacking, S.~Keel, and M.~Kontsevich, \emph{Canonical bases for
  cluster algebras}, J. Amer. Math. Soc. \textbf{31} (2018), 497--608;
  arXiv:1411.1394 [math.AG].

\bibitem[GSV08]{Gekhtman07}
M.~Gekhtman, M.~Shapiro, and A~Vainshtein, \emph{On the properties of the
  exchange graph of a cluster algebra}, Math. Res. Lett. (2008), 321--330;
  arXiv:math/0703151.

\bibitem[IIK{\etalchar{+}}10]{Inoue10c}
R.~Inoue, O.~Iyama, A.~Kuniba, T.~Nakanishi, and J.~Suzuki, \emph{Periodicities
  of {T} and {Y}-systems}, Nagoya Math. J. \textbf{197} (2010), 59--174;
  arXiv:0812.0667 [math.QA].

\bibitem[IIK{\etalchar{+}}13]{Inoue10a}
R.~Inoue, O.~Iyama, B.~Keller, A.~Kuniba, and T.~Nakanishi, \emph{Periodicities
  of {T} and {Y}-systems, dilogarithm identities, and cluster algebras {I}:
  {T}ype {$B_r$}}, Publ. RIMS \textbf{49} (2013), 1--42; arXiv:1001.1880
  [math.QA].

\bibitem[IN14]{Iwaki14a}
K.~Iwaki and T.~Nakanishi, \emph{Exact {WKB} analysis and cluster algebras}, J.
  Phys. A: Math. Theor. \textbf{47} (2014), 474009; arXiv:1401.7094 [math.CA].

\bibitem[KQ14]{Kimura12}
Y.~Kimura and Y.~Qiu, \emph{Graded quiver varieties, quantum cluster algebrs
  and dual canonical bases}, Adv. in Math. \textbf{262} (2014), 261--312;
  arXiv:1205.2066.

\bibitem[LS15]{Lee15}
K.~Lee and R.~Schiffler, \emph{Positivity for cluster algebras}, Ann. of Math.
  \textbf{182} (2015), 72--125; arXiv:1306.2416.

\bibitem[MSW11]{Musiker09}
G.~Musiker, R.~Schiffler, and L.~Williams, \emph{Positivity for cluster
  algebras from surfaces}, Adv. in Math. \textbf{227} (2011), 2241--2308;
  arXiv:0906.0748 [math.CO].

\bibitem[Nag13]{Nagao10}
K.~Nagao, \emph{Donaldson-{T}homas theory and cluster algebras}, Duke Math. J.
  \textbf{7} (2013), 1313--1367; arXiv:1002.4884 [math.AG].

\bibitem[Nak15]{Nakanishi14a}
T.~Nakanishi, \emph{Structure of seeds in generalized cluster algebras},
  Pacific J. Math. \textbf{277} (2015), 201--218; arXiv:1409.5967 [math.RA].

\bibitem[NR16]{Nakanishi15}
T.~Nakanishi and D.~Rupel, \emph{Companion cluster algebras to a generalized
  cluster algebra}, Travaux Math{\'e}matiques (2016), 129--149;
  arXiv:1504.06758 [math.RA].

\bibitem[NZ12]{Nakanishi11a}
T.~Nakanishi and A.~Zelevinsky, \emph{On tropical dualities in cluster
  algebras}, Contemp. Math. \textbf{565} (2012), 217--226, arXiv:1101.3736
  [math.RA].

\bibitem[Pla11]{Plamondon10b}
P.~Plamondon, \emph{Cluster algebras via cluster categories with
  infinite-dimensional morphism spaces}, Compos. Math. \textbf{147} (2011),
  1921--1954; arXiv:1004.0830 [math.RT].

\end{thebibliography}
\end{document}